\documentclass[12pt,a4paper,reqno]{amsart}
\usepackage{a4wide}
\usepackage[english,russian]{babel}
\usepackage[T2A]{fontenc}
\usepackage[cp1251]{inputenc} 
\usepackage{amsfonts}
\usepackage{amssymb, amsthm, amscd}
\usepackage{amsmath}
\usepackage{mathtools}
\usepackage{needspace}
\usepackage{etoolbox}
\usepackage{lipsum}
\usepackage{comment}
\usepackage{cmap}
\usepackage[pdftex]{graphicx}
\usepackage[unicode]{hyperref}
\usepackage[matrix,arrow,curve]{xy}
\usepackage[usenames,dvipsnames]{xcolor}
\usepackage{colortbl}
\usepackage{textcomp}
\usepackage{cite}
\usepackage{euscript}

\DeclareMathOperator{\id}{id}
\DeclareMathOperator{\Lie}{Lie}

\DeclareMathOperator{\ad}{ad}

\newcommand{\Z}{\mathbb{Z}}

\newcommand{\F}{\mathbb{F}}

\newcommand{\eps}{\varepsilon}

\newcommand{\sub}{\subseteq}
\newcommand{\Hom}{\mathop{\mathrm{Hom}}\nolimits}

\newcommand{\rk}{\mathop{\mathrm{rk}}\nolimits}

\renewcommand{\ge}{\geqslant}
\renewcommand{\le}{\leqslant}
\newcommand{\sm}{\setminus}
\newcommand{\map}[3]{#1\colon #2\to #3}
\newcommand{\phan}{\vphantom{,}}

\newcommand{\ox}{\otimes}
\newcommand{\restr}[2]{\left. #1\right|_{#2}}

\renewcommand{\>}{\rangle}

\newcommand{\emp}{\varnothing}

\renewcommand{\P}{\EuScript{P}}
\newcommand{\euQ}{\EuScript{Q}}
\newcommand{\euL}{\EuScript{L}}
\newcommand\leftact[2]{\phan^{#1}#2}

\theoremstyle{plain}
\newtheorem{thm}{Theorem}
\newtheorem{lem}{Lemma}
\newtheorem{prop}{Proposition}

\newtheorem{cor}{Corollary}
\newtheorem{prob}{Problem}

\theoremstyle{definition}
\newtheorem{defn}{Definition}

\theoremstyle{remark}
\newtheorem*{rem}{Remark}
\newtheorem*{warn}{Warning}

\AtBeginEnvironment{thm}{\begin{samepage}}
\AtEndEnvironment{thm}{\end{samepage}}
\AtBeginEnvironment{lem}{\begin{samepage}}
\AtEndEnvironment{lem}{\end{samepage}}
\AtBeginEnvironment{st}{\begin{samepage}}
\AtEndEnvironment{st}{\end{samepage}}
\AtBeginEnvironment{crit}{\begin{samepage}}
\AtEndEnvironment{crit}{\end{samepage}}
\AtBeginEnvironment{ax}{\begin{samepage}}
\AtEndEnvironment{ax}{\end{samepage}}
\AtBeginEnvironment{defn}{\begin{samepage}}
\AtEndEnvironment{defn}{\end{samepage}}
\AtBeginEnvironment{exmp}{\begin{samepage}}
\AtEndEnvironment{exmp}{\end{samepage}}
\AtBeginEnvironment{cor}{\begin{samepage}}
\AtEndEnvironment{cor}{\end{samepage}}
\AtBeginEnvironment{rem}{\begin{samepage}}
\AtEndEnvironment{rem}{\end{samepage}}
\AtBeginEnvironment{note}{\begin{samepage}}
\AtEndEnvironment{note}{\end{samepage}}
\AtBeginEnvironment{prop}{\begin{samepage}}
\AtEndEnvironment{prop}{\end{samepage}}

\newcommand{\Stab}{\mathop{\mathrm{Stab}}\nolimits}
\newcommand{\SL}{\mathop{\mathrm{SL}}\nolimits}

\newcommand{\D}{\mathop{\mathrm{D}}\nolimits}
\renewcommand{\tilde}{\widetilde}
\renewcommand{\hat}{\widehat}

\newcommand{\vpi}{\varpi}

\DeclareMathOperator{\gen}{gen}

\DeclareMathOperator{\lev}{lev}
\DeclareMathOperator{\ellev}{el.lev}
\DeclareMathOperator{\invlev}{inv.lev}

\makeatletter
\def\@settitle{\begin{center}%
    \baselineskip14\p@\relax
    \bfseries
    \@title
  \end{center}%
}

\def\@evenhead{\hfil\sc p. gvozdevsky\hfil}
\def\@oddhead{\hfil\sc overgroups of subsystem subgroups \hfil}
\makeatother

\def\le{\leqslant}

\def\ge{\geqslant}

\begin{document}

\title[Overgroups of subsystem subgroups]{Overgroups of subsystem subgroups in exceptional groups:\break nonideal levels}

\author{P.~Gvozdevsky}
\address{St. Petersburg\\ State University,\\ 14th Line V.O., 29,\\ Saint Petersburg 199178 Russia}
\email{gvozdevskiy96@gmail.com}
\date{}

\thanks{The author is a participant of a scientific group that won "Leader" grant by "BASIS" foundation in 2020, grant \#20-7-1-27-4. Research is also supported by «Native towns», a social investment program of PJSC ''Gazprom Neft'', and also by grant given as subsidies from Russian federal budget for creation and development of international mathematical centres, agreement between MES and PDMI RSA  № 075-15-2019-1620 from November 8, 2019.
}

\selectlanguage{english}

\maketitle

\begin{abstract}
In the present paper, a description of overgroups for the subsystem sub\-groups $E(\Delta,R)$ of the Chevalley groups $G(\Phi,R)$ over the ring $R$, where $\Phi$ is a simply laced root system and $\Delta$ is its sufficiently large subsystem, is almost entirely finished. Namely objects called levels are defined and it is shown that for any such an overgroup $H$ there exists a unique level $\sigma$ such that $E(\sigma)\le H\le \Stab_{G(\Phi,R)}(L_{\max}(\sigma))$, where $E(\sigma)$ is an elementary subgroup associated with the level $\sigma$ and $L_{\max}(\sigma)$ is a corresponding sub\-al\-gebra of the Chevalley algebra. Unlike the previous papers, here levels can be more complicated than nets of ideals.  
\end{abstract}

\section{Introduction}

In the paper \cite{Aschbacher84}, devoted to the {\it maximal subgroup project}, Michael Asch\-bacher introduced eight classes $C_1$--$C_8$ of subgroups of finite simple classical groups. The groups from these classes are "obvious" maximal subgroups of a finite classical groups. To be precise, each subgroup from an Aschbacher class either is maximal itself or is contained in a maximal subgroup that in its turn either also belongs to an Aschbacher class or can be constructed by a certain explicit procedure.

N.~A.~Vavilov defined five classes of "large" subgroups of the Chevalley groups (including exceptional ones) over arbitrary rings (see, for example, \cite{StepDiss} and \cite{VavSbgs} for details). Although these subgroups are not maximal, he conjectured that they are sufficiently large for the corresponding overgroup lattice to admit a description. One of these classes is the class of subsystem subgroups (the definition will be given in \S\ref{Systems}).

In the present paper, we continue the research on the overgroups of subsystem subgroups in exceptional groups that was started in papers \cite{VSch,GvozLevi,Gvoz2A1} and \cite{GvozInside}. This time we consider a larger class of subsystems compared to the one in paper \cite{Gvoz2A1}.

To put the results of the present paper in a context, we now recall main results that are known at the moment.

	
	$\bullet$ In \cite{BV84,KoibaevBlockdiag} and also in some other papers, the overgroups of (elementary) sub\-sys\-tem subgroups in general linear group were studied. In this case, subsystem subgroups are the groups of block-diagonal matrices.
	
	$\bullet$  Further in the thesis of N.~A.~Vavilov (see, in particular, \cite{VavMIAN} and \cite{VavSplitOrt}), this results were generalised for the cases of orthogonal and symplectic groups assuming that $2\in R^*$. The full proofs were published later. After that in the thesis of Alexander Shchegolev \cite{SchegDiss} this assumption was lifted and also the similar problem for unitary groups was solved (see also \cite{SchegMainResults} and \cite{SchegSymplectic}), which closes the case of classical groups almost entirely.
	
	$\bullet$ The cases of general linear and unitary groups can be generalised to certain noncommutative rings, but these rings should satisfy some other condition (for example, to be quasifinite or PI). The papers \cite{GolubchikSubgroups,StepOLD,VavStepSubgroupsGLStability} were devoted to such generalisations.
	
	$\bullet$ It turns out that to describe overgroups of subsystem subgroups in excep\-ti\-on\-al groups  over a commutative ring (see, for example, \cite[Pro\-blem~7]{VavStepSurvey}) is a much harder problem. The first step in this direction is a paper \cite{VSch}. The table from that paper contains the list of pairs $(\Phi,\Delta)$  for which such a description may be possible in principle, along with the number of ideals determining the level and along with certain links between these ideals.
	
	$\bullet$  In the paper \cite{GvozLevi}, author solved this problem for the subsystems $A_{l-1}\le D_l$,  $D_5\le E_6$ and $E_6\le E_7$. For the simply laced systems, these are exactly the cases where the subsystem subgroup is a Levi subgroup, and the corresponding unipotant radical is abelian. 
	
	$\bullet$ Later, in paper \cite{Gvoz2A1}, author obtained a uniform solution of this problem for the large class of subsystems in simply laced root systems. In partticular, this class includes all the subsystems of maximal rank in type $E$ root systems, for example,  $A_5+A_1\le E_6$, $A_7\le E_7$, and even $8A_1\le E_8$. In this paper and in the previous one, overgroups are described by nets of ideals of the ground ring.
	
	$\bullet$ Finally, note that the result of the paper \cite{LuzF4E6}, which describes overgroups of $F_4$ in $E_6$, is not a special case of our problem, but is pretty close to it.

Now we recall how does the answer in the results mentioned above usually look like. 


Let $G$ be an abstract group, and let $\euL$ be a certain lattice of its subgroups.  The lattice $\euL$ admits sandwich classification if it is a disjoint union of "sand\-wi\-ches":
$$
\euL=\bigsqcup_i L(F_i,N_i),\quad L(F_i,N_i)=\{H\colon F_i\le H\le N_i\},
$$
where $i$ runs through some index set, and $F_i$
is a normal subgroup of $N_i$ for all $i$. To study such a lattice, it suffices to study the quotients $N_i/F_i$. In \cite{VSch} it was conjectured that
the lattice of subgroups of a Chevalley group that contain a sufficiently large subsystem subgroup admits sandwich classification for certain $F_i$ and $N_i$. Such theorems are also called {\it the standard description}.

However, (at least) in the cases where the subsystem has an irreducible component of type $A_1$, the conjectures in \cite{VSch} should be reformulated. The reason why they cannot be true as stated is that the elementary subgroup of $\SL(2,R)$ is not normal in general.

Therefore, the main result of the present paper is similar to sandwich clas\-si\-fi\-ca\-tion, but in the present paper, as well as in the paper~\cite{Gvoz2A1}, the subgroup~$F_i$ is not normal in $N_i$ in general.

To every overgroup of the given subsystem subgroup we put in correspondence certain object called the level of this overgroup. The level defines to which sandwich the given overgroup belongs. In paper \cite{Gvoz2A1}, such objects were nets of ideals. The title of the present paper implies that this time levels can be more complicated.
 
Note that for sandwiches defined by nets of ideals one can prove normality of $F_i$ in $N_i$ provided that either the subsystem has no irreducible components of type $A_1$, or the elementary subgroup is enlarged in a certain way. This is the result of the paper \cite{GvozInside}. 
 
The paper \cite{VavGav} is devoted to the $A_2$-proof of the structure theorems. That is the proof that employs the elements of the form $x_\alpha(\xi)x_\beta(\zeta)$, where $\angle(\alpha,\beta)=\pi/3$ to get into a parabolic subgroup. In paper \cite{Gvoz2A1}, author used a method based on the remark after the proof of the main lemma in \cite{VavGav}. According to that remark, to get into a parabolic subgroup one can also use the elements of the form $x_\alpha(\xi)x_\beta(\zeta)$, where $\angle(\alpha,\beta)=\pi/2$. Such method should be called the $2A_1$-proof, and it allows us to study the overgroups of subsystem subgroups even if the subsystem has type $nA_1$. In the present paper both methods are involved.

The paper is organised as follows. In \S\S\ref{SecBasicNotation} and \ref{SecSpecificNotation} we introduce the necessary notation. In \S\S\ref{SecOveralgebras}--\ref{SecLelelComp} we prove some preparatory lemmas. In \S\ref{SecWTF} we sum up all that was said in the previous sections, state the main result of the present paper, and make a plan of its proof. In \S\S\ref{Sec2A1} and \ref{SecA2}, we finish the proof of the main result. In \S\ref{SecSpecialCases} we consider important special cases and corollaries of the main theorem. 

\smallskip

I am grateful to my teacher N.~A.~Vavilov for setting the problem and for extremely
helpful suggestions.

\section{Basic Notation}
\label{SecBasicNotation}

\subsection{Root systems and Chevalley groups}
\label{Systems}

Let $\Phi$ be an irreducible root system in the sense of \cite{Bourbaki4-6}, $\P$ a lattice that is intermediate between the root lattice $\euQ(\Phi)$ and the weight lattice $\P(\Phi)$, $R$ a commutative associative ring with unity, $G(\Phi,R)=G_\P(\Phi,R)$ a Chevalley group of type $\Phi$ over $R$, $T(\Phi,R)=T_\P(\Phi,R)$ a split maximal torus of $G(\Phi,R)$. For every root $\alpha\in\Phi$ we denote by $X_\alpha=\{x_\alpha(\xi),\colon \xi\in R\}$ the corresponding unipotent root subgroup with respect to $T$. We denote by $E(\Phi,R)=E_\P(\Phi,R)$ the elementary subgroup generated by all $X_\alpha$, $\alpha\in\Phi$. 

In the rest of the paper we always assume that $\Phi$ is simply laced unless stated otherwise. 

Let $\Delta$ be a subsystem of $\Phi$. We denote by $E(\Delta,R)$ the subgroup of $G(\Phi,R)$, generated by all~$X_\alpha$, where $\alpha\in \Delta$. It is called an (elementary) {\it subsystem subgroup}. It can be shown that it is an elementary subgroup of a Chevalley group $G(\Delta,R)$, embedded into the group $G(\Phi,R)$. Here the lattice between $\euQ(\Delta)$ and $\P(\Delta)$ is an orthogonal projection of $\P$ onto the corresponding subspace.

We denote by $W(\Phi)$ resp. $W(\Delta)$ the Weyl groups of the systems $\Phi$ resp. $\Delta$.

We are going to describe intermediate subgroups between $E(\Delta,R)$ and $G(\Phi,R)$. The pair $(\Phi,\Delta)$ here should satisfy certain condition, which we for\-mu\-late in Subsection \ref{Condition}.

\subsection{Affine schemes}

The functor $G(\Phi,-)$ from the category of rings to the category of groups is an affine group scheme (a Chevalley--Demazure scheme). This means that its composition with the forgetful functor to the category of sets is representable, i.e.,
$$
G(\Phi,R)=\Hom (\Z[G],R).
$$
The ring $\Z[G]$ here is called the {\it ring of regular functions} on the scheme $G(\Phi,-)$. 

An element $g_{\gen}\in G(\Phi,\Z[G])$ that corresponds to the identity ring ho\-mo\-mo\-rphism is called the {\it generic element} of the scheme $G(\Phi,-)$. This element has a universal property: for any ring $R$ and for any $g\in G(\Phi,R)$, there exists a unique ring homomorphism
$$
\map{f}{\Z[G]}{R},
$$
such that $f_*(g_{\gen})=g$. For details about application of the method of generic elements
to the problems similar to that of ours, see the paper of A.~V.~Ste\-pa\-nov~\cite{StepUniloc}.

\subsection{Group theoretic notation}
\begin{itemize}
	
	\item If the group $G$ acts on the set  $X$, $x\in X$ and $Y\sub X$, we denote by $\Stab_G(x)$ resp. $\Stab_G(Y)$ the stabiliser
	of the element~$x$ resp. the stabiliser of the subset $Y$ (as a subset, not pointwise).
	
	\item Commutators are left normalised:
	$$
	[x,y]=xyx^{-1}y^{-1}.
	$$
	
	\item Upper index stands for the left or right conjugation: 
	$$
	\leftact{g}{h}=ghg^{-1},\quad h^g=g^{-1}hg.
	$$
	
	\item If $X$ is a subset of the group $G$, we denote by $\<X\>$ the subgroup generated by $X$.
	
\end{itemize}

\subsection{Lie algebras}

Let us introduce some notation concerning Lie algebras.

We denote by $L(\Phi,\Z) $ the integer span of the Chevalley basis in the complex Lie algebra of type $\Phi$ (see~\cite{Humphreys}). The following notation stands for Chevalley algebra
 $L(\Phi,R)=L(\Phi,\Z)\ox_\Z R$. This is a Lie algebra over the ring $R$ equipped with an action of the group $G(\Phi,R)$ called the adjoint representation. For elements $g\in G(\Phi,R)$ and $v\in L(\Phi,R)$, we denote this action by $\leftact{g}{v}$.

Note that the algebra $L(\Phi,R)$, in general, is not isomorphic to the tangent Lie algebra of the algebraic group $G(\Phi,R)$ (see, for example, \cite{RoozemondDiss}). However, if the group is simply connected, then these algebras are canonically isomorphic.

\smallskip

We denote by $e_\alpha$, $\alpha\in\Phi$ and $h_i$, $i=1,\ldots,\rk\Phi$ the Chevalley basis of the Lie algebra $L(\Phi,R)$. By $D$ we denote its toric subalgebra generated by~$h_i$.

For every element $v\in L(\Phi,R)$, we denote by $v^\alpha$ and $v^i$ its coefficient in Chevalley basis. 

Lie bracket is denoted by $[\,\cdot\,,\,\cdot\,]$. The same notation stands for the group commutator, but it should be clear from the context whether the calculation are in a group or in a algebra. 

\section{Specific notation}
\label{SecSpecificNotation}

In this section we introduce some specific notation and prove lemmas related to it.

\subsection{Condition on subsystem, blocks, and combinatorial lemma}
\label{Condition}

Let $\Delta$ be a subsystem of the root system $\Phi$. Let $\Delta'$ be a subsystem of the root system $\Phi'$. Both systems $\Phi$ and $\Phi'$ are supposed to be simply laced. Assume that there is an embedding $\Phi'\hookrightarrow \Phi$ as subsystem such that $\Delta\cap\Phi'=\Delta'$ and any root from $\Delta\sm\Delta'$ is orthogonal to any root from $\Phi'$.

It is easy to see that under assumptions above the problem of overgroup description for the subsystem $\Delta\le\Phi$ in some sense includes the similar problem for the subsystem $\Delta'\le\Phi'$. Namely there is an obvious inclusion of the second overgroup lattice into the first one. We will denote such situation by phrase ``the problem for the subsystem $\Delta\le\Phi$ includes the problem for the subsystem $\Delta'\le\Phi'$\,''.

The subsystem $\Delta\le\Phi$ will be supposed to satisfy the following restriction.
\begin{equation}
	\begin{split}
		&\text{The problem for the subsystem } \Delta\le\Phi \text{ does not include}\\ &\text{ problems for subsystems of types } 
		\emp\le A_1 \text{ and } A_1\le A_2.
	\end{split}\tag{$*$}
\end{equation}

The first part of the restriction means just that the system $\Phi$ dose not have roots orthogonal to the whole subsystem $\Delta$.

Note that this restriction is necessary. There are no any reasonable ways to describe the subgroup lattice of the group $\SL(2,R)$. As for the overgroups of $\SL(2,R)$ in $\SL(3,R)$, we do not know any methods that could allow to obtain such description, and probably there is none.

From this moment we will always assume that the condition $(*)$ holds true. 

We say that the roots $\gamma_1$, $\gamma_2$ are equivalent if for any root $\alpha\in\Delta$ we have $(\gamma_1,\alpha)=(\gamma_2,\alpha)$. The equivalence class with respect to this relation will be called {\it block}. Note that any root from the subsystem $\Delta$ is a unique element in its block.  

Subsystems that satisfy the condition $(*)$ have the following properties, which we will use.

\begin{lem}\label{Roots} Assume that subsystem $\Delta\le\Phi$ satisfies the condition $(*)$. Then
	\begin{enumerate}
		\item For any root $\gamma\in\Phi\sm\Delta$ there exists a root $\alpha_1\in\Delta$ such that we have $(\gamma,\alpha_1)=-1$.
		
		\item For any roots $\gamma\in\Phi\sm\Delta$ and $\alpha_1\in\Delta$ such that $(\gamma,\alpha_1)=-1$ there exists a root $\alpha_2\in\Delta$ such that $(\gamma,\alpha_2)=-1$ and either $\alpha_1\perp\alpha_2,$ or $(\alpha_1,\alpha_2)=1$.
		
		\item Any block consists of pairwise orthogonal roots.
		
		\item Any block contains not more than three roots.
\end{enumerate} 
\end{lem}

\begin{proof}
	\begin{enumerate}
		\item It follows from the first part of the condition $(*)$ that there exists a root $\alpha_1\in\Delta$ such that $(\gamma,\alpha_1)\ne 0$. Replacing, if necessary, $\alpha_1$ by $-\alpha_1$, we may assume that $(\gamma,\alpha_1)=-1$.
		
		\smallskip
		
		\item  It follows from the second part of the condition $(*)$ that there exists a root $\beta\in\Delta\sm{\pm\alpha_1}$ such that it is not orthogonal at least to one of the roots $\gamma$ and $\alpha_1$. Replacing, if necessary, $\beta$ by $-\beta$, we may assume that its inner product to one of the roots $\gamma$ and $\alpha_1$ is equal to $-1$. Set $\alpha_2=\beta$ if $(\gamma,\beta)=-1$, and $\alpha_2=\beta+\alpha_1$ if $(\gamma,\beta)=0$, but $(\alpha_1,\beta)=-1$.
		
		In any case we have $(\gamma,\alpha_2)=-1$ and $\alpha_1\ne \pm\alpha_2$. It remains to notice that it is impossible that $(\alpha_1,\alpha_2)=-1$. Indeed, in this case  we obtain that $\gamma=-\alpha_1-\alpha_2\in\Delta$. Therefore, we have either $\alpha_1\perp\alpha_2$, or $(\alpha_1,\alpha_2)=1$.
		
		\smallskip
		
		\item Let $\gamma_1$,$\gamma_2\in\Phi$ be distinct equivalent roots. Clearly, they can not belong to the subsystem $\Delta$.
		
		They also can not be opposite to each other. Indeed, let us assume the converse. By the first item there exists a root $\alpha\in\Delta$ such that $(\gamma_1,\alpha)=-1$. Then $(\gamma_2,\alpha)=1\ne(\gamma_1,\alpha)$.
		
		It is impossible that $(\gamma_1,\gamma_2)=1$. Indeed, let us assume the converse and consider two cases. If $\gamma_1-\gamma_2\in\Delta$, then this difference is a root from $\Delta$ that has different inner products with roots $\gamma_1$ and $\gamma_2$. If $\gamma_1-\gamma_2\notin\Delta$, then by first item there exists a root $\alpha\in\Delta$ such that $(\gamma_1-\gamma_2,\alpha)=-1$; hence $(\gamma_2,\alpha)\ne(\gamma_1,\alpha)$.
		
		Finally, it is impossible that $(\gamma_1,\gamma_2)=-1$. Indeed, let us assume the converse. Then for any root $\alpha\in\Delta$ such that $(\gamma_1,\alpha)=-1$, we also have $(\gamma_2,\alpha)=-1$; hence $\alpha=-\gamma_1-\gamma_2$, i.e. there is only one such root. However, by first two items there are at least two such roots. Therefore, we have $\gamma_1\perp\gamma_2$.
		
		\smallskip
		
		\item Assume that $\gamma_1$,$\gamma_2$,$\gamma_3$,$\gamma_4\in\Phi$ are pairwise distinct roots. Clearly, they can not belong to the subsystem $\Delta$. By third item they are pairwise orthogonal. For any root $\alpha\in\Delta$ such that $(\gamma_1,\alpha)=-1$ we also have $(\gamma_i,\alpha)=-1$ for $i=2,3,4$; hence $\alpha=-\tfrac{1}{2}(\gamma_1+\gamma_2+\gamma_3+\gamma_4)$, i.e. there is only one such root. However, by first two items there are at least two such roots. \qedhere		 
	\end{enumerate}
\end{proof}

\begin{rem} One can show that blocks of three roots occur if and only if the problem for the subsystem $\Delta\le\Phi$ includes the problem for the subsystem $A_2\le\D_4$.
\end{rem}

For any root $\gamma\in\Phi$ we denote by $[\gamma]$ the block that contains~$\gamma$. For a root $\alpha\in\Delta$ set $([\gamma],\alpha)=(\gamma,\alpha)$. This is well defined by the definition of block. By $|[\gamma]|$ we denote the number of roots in the block~$[\gamma]$.

\subsection{Lie algebras}

For a block $[\gamma_1]=\{\gamma_i\}_{i=1}^{|[\gamma_1]|}$ we consider the following $R$\nobreakdash-submodule in the Lie algebra $L(\Phi,R)$:
$$
M_{[\gamma_1]}=\bigoplus_i R\cdot e_{\gamma_i}\le L(\Phi,R).
$$

Therefore, we have
$$
L(\Phi,R)=D\oplus\bigoplus_{[\gamma]} M_{[\gamma]}.
$$

For an element $v\in L(\Phi,R)$ we denote by $v^{[\gamma]}$ its component in the module $M_{[\gamma]}$; by $v^d$ we denote its component in Cartan subalgebra $D$.

\subsection{\!Prelevels}

\!Assume that a block $[\gamma_1]\!=\!\{\gamma_i\}_{i\!=\!1}^{|[\gamma_1]|}$ and a root ${\alpha\!\in\!\Delta,}$ are such that $([\gamma_1],\alpha)\!=\!-\!1$. Then the set $[\gamma_1]+\alpha$, i.e. the set $\{\gamma_i+\alpha\}_{i=1}^{|[\gamma_1]|}$, is also a block because it is obtained from the block $[\gamma_1]$ by reflection with respect to the root $\alpha$. Restriction of the adjoint action of the element~$e_\alpha$,
$$
	\map{\ad e_\alpha}{L(\Phi,R)}{L(\Phi,R)},\quad
	x\mapsto [e_\alpha,x],
$$
defines an $R$-linear map
$$
\map{T_{[\gamma_1]\to[\gamma_1]+\alpha}}{M_{[\gamma_1]}}{M_{[\gamma_1]+\alpha}}.
$$

The following definition is the first step to defining objects that classify overgroups in our problem.

\begin{defn}
\label{prelevels} Assume that we are given a collection $\sigma=\{\sigma_{[\gamma]}\}$, where $[\gamma]$ runs through the set of all blocks, and $\sigma_{[\gamma]}$ is an $R$-submodule of the module $M_{[\gamma]}$. Such a collection is called a {\it prelevel} if the following two conditions hold true:

		\noindent(1) for any block $[\gamma]$ and any root $\alpha\in\Delta$ such that $([\gamma],\alpha)=-1$ we have $\sigma_{[\gamma]+\alpha}=T_{[\gamma]\to[\gamma]+\alpha}(\sigma_{[\gamma]})$;
		
		\noindent(2)  for any root $\alpha\in\Delta$ we have $\sigma_{[\alpha]}=M_{[\alpha]}=R\cdot e_{\alpha}$.
	\end{defn}

We will use the following property of the operators $T_{[\gamma]\to[\gamma]+\alpha}$.

\begin{lem}\label{BackAndForth} Let a block $[\gamma]$ and a root $\alpha\in\Delta$ are such that $([\gamma],\alpha)=-1$. Then we have 
	$$
	T_{[\gamma]+\alpha \to[\gamma]}\circ T_{[\gamma]\to[\gamma]+\alpha}=\id_{M_{[\gamma]}}.
	$$
	\vspace{-4ex}
	\end{lem}
\begin{proof}
	By Jacobi identity we have
	$$
	\ad e_{-\alpha}\circ \ad e_\alpha= -\ad h_{\alpha}+\ad e_{\alpha}\circ \ad e_{-\alpha}.
	$$
	Since the restriction of the operator $\ad e_{-\alpha}$ to the submodule $M_{[\gamma]}$ vanishes, we have
	$$
	T_{[\gamma]+\alpha \to[\gamma]}\circ T_{[\gamma]\to[\gamma]+\alpha}=\restr{-\ad h_{\alpha}}{M_{[\gamma]}}
	$$
	The operator $\ad h_{\alpha}$ acts on the vector $e_{\gamma}$ via multiplication by $-1$, and the same holds true for all roots in the block $[\gamma]$. Therefore, we obtain that
	\begin{equation*}
	T_{[\gamma]+\alpha \to[\gamma]}\circ T_{[\gamma]\to[\gamma]+\alpha}=\id.\qedhere
	\end{equation*}
\end{proof}

\section{Lie overalgebras}
\label{SecOveralgebras}

In this section we describe overalgebras of a subsystem Lie subalgebra. This is the first step towards the description of overgroups. 

\smallskip

Consider an $R$-subalgebra $L'(\Delta,R)\le L(\Phi,R)$ generated by elements $e_{\alpha}$, where $\alpha\in\Delta$. For any $R$-overalgebra $L'(\Delta,R)\le L\le L(\Phi,R)$ we define its level $\lev(L)=\{\lev(L)_{[\gamma]}\}$, where
$$
\lev(L)_{[\gamma]}=M_{[\gamma]}\cap L. 
$$
  
  It is easy to see that $\lev(L)$ is a prelevel. 
  
  We will use the next lemma in order  to describe all $R$-overalgebras of the algebra $L'(\Delta,R)$.
   
   \begin{lem}
	\label{graded} For any $R$-overalgebra $L'(\Delta,R)\le L\le L(\Phi,R)$ the following equality holds true 
	$$
	L=(D\cap L)\oplus\bigoplus_{[\gamma]}\lev(L)_{[\gamma]}.
	$$
\end{lem}
\begin{proof}
	The right hand side is contained in the left hand side by definition. Let us prove the opposite inclusion. Assume the converse, i.e there exists an element $v\in L$ that does not belong to the right hand side. We may assume that among all such elements the element $v$ is chosen in such a way that the number of blocks $[\gamma]$ such that components $v^{[\gamma]}$ are nonzero is the smallest. Clearly, any such component $v^{[\gamma]}$ does not belong to the submodule $\lev(L)_{[\gamma]}$ because otherwise we would subtract it and reduce the number of nonzero components. In particular, we have $\gamma\notin\Delta$. 
	
	If all the components $v^{[\gamma]}$ vanish, then we have $v\in D$. Hence $v\in D\cap L$, i.e. the element $v$ belongs to the right hand side. Therefore, there is at least one block $\gamma$ such that $v^{[\gamma]}\ne 0$.
	
	If there are at least two such blocks, say $[\gamma_1]$ and $[\gamma_2]$, then since these blocks are distinct, it follows that there exists a root $\alpha\in\Delta$ such that $([\gamma_1],\alpha)\ne ([\gamma_2],\alpha)$. Replacing, if necessary, the root $\alpha$ to its opposite and swapping, if necessary, the roots $\gamma_1$ and $\gamma_2$, we may assume that $([\gamma_1],\alpha)=-1$ and $([\gamma_2],\alpha)\ne -1$. Then taking the element
	$$
	v'=[e_\alpha,v]-[e_{\alpha},v]^\alpha e_\alpha,
	$$ 
	we reduce the number of nonzero components. In addition, we have
	$$
	(v')^{[\gamma_1]+\alpha}=T_{[\gamma_1]\to[\gamma_1]+\alpha}(v^{\gamma})\notin T_{[\gamma_1]\to[\gamma_1]+\alpha}(\lev(L)_{[\gamma]})=\lev(L)_{[\gamma]+\alpha}.
	$$ 
	Hence the element $v'$ does not belong to the right hand side of the equality in question.
	
	It remains to consider the case, where the nonzero component $v^{[\gamma]}$ is unique. Then $v=v^{[\gamma]}+v^d$. Choose a root $\alpha\in\Delta$ such that $([\gamma],\alpha)=-1$. Taking the element
	$$
	v'=[e_\alpha,v]-[e_{\alpha},v]^\alpha e_\alpha,
	$$ 
	we may assume that $v^d=0$. Therefore, we have 
	$$
	v\in M_{[\gamma]}\cap L=\lev(L)_{[\gamma]};
	$$
	hence the element $v$ belongs to the right hand side of the equality in question.
\end{proof}

The following definition is the second step to defining objects that classify overgroups in our problem. 

\begin{defn}
\label{AlmostLevel} A prelevel $\sigma$ is called an {\it almost level} if there exists an $R$-over\-algebra $L'(\Delta,R)\le L\le L(\Phi,R)$ such that $\lev(L)=\sigma$.
\end{defn}

For example, it is easy to see that in case, where all the blocks have one root, almost levels corresponds to nets of ideals (see definition in \cite{Gvoz2A1}).

\smallskip

The next proposition describes all $R$-overalgebras of the algebra 
$
L'(\Delta,R).
$
\begin{prop}
\label{Overalgebras} Let $\sigma$ be an almost level. Then the following statements hold true.
\begin{enumerate}
\item Among all $R$-overalgebras of the algebra $L'(\Delta,R)$ of level $\sigma$ there exist the largest one and the smallest one by inclusion. We denote them by $L_{\min}(\sigma)$ and $L_{\max}(\sigma)$ correspondently\textup;

\item For any $R$-overalgebra $L'(\Delta,R)\le L\le L(\Phi,R)$ of level $\sigma$ the algebra $L_{\max}(\sigma)$ coincides with the normalizer of the algebra $L,$ i.e.
$$
L_{\max}(\sigma)=\{v\in L(\Phi,R)\mid [v,L]\sub L\}.
$$
\end{enumerate}
\end{prop}

\begin{proof}
Clearly, the smallest by inclusion overalgebra of level $\sigma$ is an algebra generated by all the elements $\sigma_{[\gamma]}$. Further let us take an arbitrary overalgebra $L'(\Delta,R)\le L\le L(\Phi,R)$ of level $\sigma$ and show that its normalizer, which we temporarily denote by $N$, is the biggest by inclusion overalgebra of level $\sigma$. Therefore, we prove the second part of the first statement and the whole second statement simultaneously.

Firstly, normalizer is always closed with respect to Lie bracket; this well known fact follows easily from Jacobi identity. Secondly, the inclusion $L\le N$ holds; hence the algebra $N$ is an overalgebra of the algebra $L'(\Delta,R)$, and we have $\sigma\le\lev(N)$. Let us prove the opposite inclusion: $\lev(N)\le\sigma$. Let $v\in \lev(N)_{[\gamma]}$. Choose a root $\alpha\in\Delta$ such that $([\gamma],\alpha)=-1$; then we have:
\begin{align*}
T_{[\gamma]\to [\gamma]+\alpha}(v)&=[e_{\alpha},v]\in [L,N]\cap M_{[\gamma]+\alpha}\le L\cap M_{[\gamma]+\alpha}
\\
& = \sigma_{[\gamma]+\alpha}= T_{[\gamma]\to [\gamma]+\alpha}(\sigma_{[\gamma]}).
\end{align*}
Hence $v\in\sigma_{[\gamma]}$. Therefore, $\lev(N)=\sigma$.

Finally, let $L_1$ be another overalgebra of level $\sigma$; let us prove that $L_1\le N$. By Lemma \ref{graded} we have: 
$$
L=(D\cap L)\oplus\bigoplus_{[\gamma]}\sigma_{[\gamma]}, \quad
L_1=(D\cap L_1)\oplus\bigoplus_{[\gamma]}\sigma_{[\gamma]}.
$$ 
We already know that
$$
\bigoplus_{[\gamma]}\sigma_{[\gamma]}\le N.
$$
It remains to check that $D\cap L_1\le N$. Since $[D,D]=0$, it is enough to check that $[D\cap L_1,\sigma_{[\gamma]}]\le \sigma_{[\gamma]}$; but that follows from the fact that $L_1$ is closed with respect to Lie bracket.
\end{proof}

Let $\sigma$ be an almost level. Consider the stabilizer of the subalgebra $L_{\max}(\sigma)$ with respect to the adjoin action of the group $G(\Phi,R)$:
$$
S(\sigma)=\Stab_{G(\Phi,R)}L_{\max}(\sigma).
$$
It is easy to see that in case, where all the blocks have one root, this notation agrees with the similar one given in the paper \cite{Gvoz2A1}.

From the second item of Proposition \ref{Overalgebras} we immediately deduce the following statement.

\begin{cor}
\label{LargestStabilaser} For any $R$-overalgebra $L'(\Delta,R)\le L\le L(\Phi,R)$ of level $\sigma$ we have the following inclusion
$$
\Stab_{G(\Phi,R)}(L)\le S(\sigma).
$$
\end{cor}

\section{Tandems}
\label{SecTandems}

The notion of a tandem was introduced in the paper \cite{Gvoz2A1}. Let us recall the definition and the properties that we will use. 

\begin{defn} We call an element of the set $G(\Phi,R)\times L(\Phi,R)$ a {\it tandem} if it can be written as $(\leftact{h}{(x_\alpha(\xi))},\leftact{h}{(\xi e_\alpha)})$, where $\alpha\in\Phi$, $\xi\in R$ and $h\in G(\Phi,R)$.
\end{defn}

Actually, it can be shown that the first component of a tandem can be recovered from the second one, but this is inessential for our purposes.

\begin{rem}
	\label{rootchange} 
	For a given root $\beta\in\Phi$ any tandem $(\leftact{h}{(x_\alpha(\xi))},\leftact{h}{(\xi e_\alpha)})$ can be written as $(\leftact{h'}{(x_\beta(\xi'))},\leftact{h'}{(\xi' e_\beta)})$. Moreover, $h'$ can be obtain from $h$ by multiplication by certain element of the extended  Weyl group, and $\xi'=\pm\xi$.
\end{rem}

The proofs of the following lemmas, which help perform calculations with tandems, are written in the paper \cite{Gvoz2A1}. The last one was stated in \cite{Gvoz2A1} not for the arbitrary parabolic subset of roots, but only for those considered in that paper; however, the proof works for the arbitrary one.

\begin{lem}
\label{tandemactpre} The first component of a tandem acts on the element $v\in L(\Phi,R)$ in a following way\textup:
	$$
	\leftact{(\leftact{h}{x_{\alpha}(\xi)})}{v}=v+[\leftact{h}{(\xi e_\alpha)},v]-\xi(\leftact{h^{-1}}{v})^{-\alpha}\cdot\leftact{h}{(\xi e_{\alpha})}.
	$$
	\end{lem}

\begin{lem}
\label{tandemact} Let $(g,l)$ be a tandem, and $\beta\in\Phi$. Then
	$$
	\leftact{g}{e_\beta}=e_\beta+[l,e_\beta]-l^{-\beta} \cdot l.
	$$
\end{lem}

\begin{lem}
\label{tandemsinparabolic} Let $\mathfrak{S}\sub\Phi$ be a parabolic subset of roots. Let $P_{\mathfrak{S}}$ be the corresponding parabolic subgroup. Let $(g,l)$ be a tandem such that $l$ belongs to the Lie algebra $L_{\mathfrak{S}}$, where
	$$
	L_{\mathfrak{S}}=D\oplus\bigoplus_{\alpha\in\mathfrak{S}} R\cdot e_\alpha\le L(\Phi,R).
	$$
	Then $g\in P_{\mathfrak{S}}$.
	\end{lem}

\begin{warn}
 If the group $G(\Phi,R)$ is not simply connected, then the Lie algebra $L_{\mathfrak{S}}$ is not the same as the tangent Lie algebra $\Lie(P_{\mathfrak{S}})$ of the subgroup~$P_{\mathfrak{S}}$.
\end{warn}

\section{The condition on the ground ring}
\label{SecCOnditionOnRing}

For some cases we have to impose the following restriction on the ring $R$. 
 \begin{equation}
\begin{split}
\text{The ring } R \text{ does not have a residue field } \F_2.
\end{split}\tag{$**$}
\end{equation}
From this moment we will always assume that if there is a block of three roots, then the condition $(**)$ holds true. Later we will state more precisely in which cases this condition is required. 

\smallskip

The following lemma holds for rings that satisfy the above condition. 

\begin{lem}
\label{EliminateSquareTerm} Assume that the condition $(**)$ holds true. Let $M$ be a module over the ring $R$; and let $x,$$y\in M$. The the submodule generated by all the elements $tx+t^2y,$ where $t$ runs through $R,$ contains the element $x$.
\end{lem}
\begin{proof}
Let us denote this submodule by $M_0$. Taking $t=-1$ we obtain $x-y\in M_0$. Further for an arbitrary $t\in R$ we have
$$
(t+t^2)x=tx+t^2y+t^2(x-y)\in M_0.
$$
It remains to notice that the elements $(t+t^2)\in R$ generate the unit ideal. Indeed, otherwise, the ideal generated by them is contained in some maximal ideal, and in the corresponding residue field all elements are roots of the equation $t^2+t=0$, i.e this residue field is $\F_2$. 
\end{proof}

\section{Level computation}
\label{SecLelelComp}

In this section we introduce two definitions for the level of an overgroup. Essential part of the present paper is the proof of their coincidence. 

\subsection{Elementary level}

Consider a block $[\gamma_1]=\{\gamma_i\}_{i=1}^{|[\gamma_1]|}$, and an element $a=\sum_i \xi_ie_{\gamma_i} \in M_{[\gamma_1]}$. Set
$$
x_{[\gamma_1]}(a)=\prod_i x_{\gamma_i}(\xi_i).
$$
This definition does not depend on the order of the roots $\gamma_i$ because by Lemma~\ref{Roots} these roots are orthogonal.

For any overgroup $E(\Delta,R)\le H\le G(\Phi,R)$ we define its {\it elementary level} as follows: $\ellev(H)=\{\ellev(H)_{[\gamma]}\}$, where
$$
\ellev(H)_{[\gamma]}=\{a\in M_{[\gamma]}\colon x_{[\gamma]}(a)\in H\}. 
$$

Let us prove the following proposition.

\begin{prop}
\label{EllevIsPrelev} For any overgroup $E(\Delta,R)\le H\le G(\Phi,R)$ its elementary level $\ellev(H)$ is a prelevel.
\end{prop}
\begin{proof}
	For any root $\gamma\in \Phi$ we set
	$$
	I_{\gamma}=\{\xi\in R\colon x_{\gamma}(\xi)\in H\}.
	$$
	It follows from the paper \cite{VSch} that the sets $I_\gamma$ are ideals of the ring $R$ and that the ideal $I_\gamma$ depend only on the orbit of the root $\gamma$ under the action of the Weyl group $W(\Delta)$.
	
	Now we get to the proof. Firstly, the subsets $\ellev(H)_{[\gamma]}\sub M_{[\gamma]}$ are clearly additive subgroups. Secondly, it is also clear that for any root $\alpha\in\Delta$ we have 
	$$
	\ellev(H)_{[\alpha]}=M_{[\alpha]}.
	$$
	
	Further let us assume that the block $[\gamma]$ and the root $\alpha\in\Delta$ are such that $([\gamma],\alpha)=-1$. We prove the following inclusion
	$$
	R\cdot T_{[\gamma]\to [\gamma]+\alpha}(\ellev(H)_{[\gamma]})\le \ellev(H)_{[\gamma]+\alpha}.
	$$ 
	It is easy to see that this will be enough to finish the proof of the proposition (follows from the first item of Lemma \ref{Roots} and Lemma \ref{BackAndForth}).
	
	As we remember from Lemma \ref{Roots}, the block $[\gamma]$ may contain one, two, or three roots.
	
	\smallskip
	
	\noindent{\bf Case 1:} $|[\gamma]|=1$.
	
	In this case we have 
	$$
	\ellev(H)_{[\gamma]}=I_{\gamma}e_{\gamma},
	$$
	and similarly 
	$$
	\ellev(H)_{[\gamma]+\alpha}=I_{\gamma+\alpha}e_{\gamma+\alpha}.
	$$
	The statement in question now follows from the reference above to the paper~\cite{VSch}.
	
	\smallskip
	
	 \noindent{\bf Case 2:} $|[\gamma]|=2$.
	 
	 Let $[\gamma]=\{\gamma_1,\gamma_2\}$. As we remember from Lemma \ref{Roots}, the roots $\gamma_1$ and $\gamma_2$ are orthogonal. Take an element $a=\xi_1e_{\gamma_1}+\xi_2e_{\gamma_2}\in \ellev(H)_{[\gamma]}$ and an element $\zeta\in R$. First note that $\xi_1\xi_2\in I_{\gamma_1+\gamma_2+\alpha}$. Indeed, applying Chevalley commutator formula we get
	 $$
	 \big[x_{-\alpha}(1),[x_{\alpha}(1),x_{[\gamma]}(a)]\big]=x_{[\gamma]}(a)x_{\gamma_1+\gamma_2+\alpha}(\pm\xi_1\xi_2).
	 $$
	Here we used Lemma \ref{BackAndForth} in order to determine signs that have components of the element $a$ in the right hand side of the equality. Further since the left hand side and the element $x_{[\gamma]}(a)$ belong to the subgroup $H$, it follows that $x_{\gamma_1+\gamma_2+\alpha}(\pm\xi_1\xi_2)\in H$, i.e. $\xi_1\xi_2\in I_{\gamma_1+\gamma_2+\alpha}$.
	
	Further consider the equality
	$$
	[x_{\alpha}(\zeta),x_{[\gamma]}(a)]=x_{[\gamma]+\alpha}\left(\zeta T_{[\gamma]\to[\gamma]+\alpha}(a)\right)x_{\gamma_1+\gamma_2+\alpha}(\pm\zeta\xi_1\xi_2).
	$$
	Its left hand side belongs to the subgroup $H$; and since $I_{\gamma_1+\gamma_2+\alpha}$ is an ideal, it follows that $x_{\gamma_1+\gamma_2+\alpha}(\pm\zeta\xi_1\xi_2)\in H$. Hence we have $x_{[\gamma]+\alpha}\left(\zeta T_{[\gamma]\to[\gamma]+\alpha}(a)\right)\in H$, and $\zeta T_{[\gamma]\to[\gamma]+\alpha}(a)\in \ellev(H)_{[\gamma]+\alpha}$.
	
	\smallskip
	
	\noindent{\bf Case 3:} $|[\gamma]|=3$.
	
	Let us recall that if there is a block of three roots, then we assume the condition $(**)$.
	
	Let $[\gamma]=\{\gamma_1,\gamma_2,\gamma_3\}$. Note that in this case the roots $\gamma_1+\gamma_2+\gamma_3+\alpha$ and $\gamma_1+\gamma_2+\gamma_3+2\alpha$ belong to the subsystem $\Delta$. Indeed, by item2 of Lemma \ref{Roots}, there is a root $\beta\in\Delta$ such that $([\gamma],\beta)=-1$ and either $\alpha\perp\beta$, or $(\alpha,\beta)=1$. The first case is impossible because in this case the matrix of the inner products of the roots $\gamma_1$,$\gamma_2$,$\gamma_3$,$\alpha$,$\beta$ is indefinite. Hence we have $(\alpha,\beta)=1$. Then the mentioned above matrix of inner products is positively semi-definite; and we can point out the dependence between these roots: 
	$$
	\beta+\gamma_1+\gamma_2+\gamma_3+\alpha=0.
	$$
	Thus we have 
	$$
	\gamma_1+\gamma_2+\gamma_3+\alpha=-\beta\in\Delta\quad\text{and}\quad \gamma_1+\gamma_2+\gamma_3+2\alpha=\alpha-\beta\in\Delta.
	$$ 
	
Further we set:
	\begin{align*}
	A&=\{v\in R^6\colon\forall\,t\in R\, x_{\gamma_1+\alpha}(tv_1)
	\\
	&\times\, x_{\gamma_2+\alpha}(tv_2)x_{\gamma_3+\alpha}(tv_3)x_{\gamma_1+\gamma_2+\alpha}(tv_4)
	 x_{\gamma_1+\gamma_3+\alpha}(tv_5) x_{\gamma_2+\gamma_3+\alpha}(tv_6)\in H\}.
	\end{align*}
	The set $A$ is an additive subgroup in $R^6$ because the root elements in its definition commute modulo the root subgroup of the root $\gamma_1+\gamma_2+\gamma_3+2\alpha$, which is contained in the subgroup $H$. It is also clear that $A=RA$, i.e. it is an  $R$-submodule. 
	
	Take an element $a=\xi_1e_{\gamma_1}+\xi_2e_{\gamma_2}+\xi_3e_{\gamma_3}\in \ellev(H)_{[\gamma]}$ and elements $s$,$t\in R$.	
	
	From the equality
	\begin{align*}
	[x_{\alpha}(1),x_{[\gamma]}(a)]&=x_{[\gamma]+\alpha}\left(T_{[\gamma]\to[\gamma]+\alpha}(a)\right)
	\\
	&\times\,x_{\gamma_1+\gamma_2+\alpha}(\pm\xi_1\xi_2)x_{\gamma_1+\gamma_3+\alpha}(\pm\xi_1\xi_3)x_{\gamma_2+\gamma_3+\alpha}(\pm\xi_2\xi_3)
	\\
	&\times\,x_{\gamma_1+\gamma_2+\gamma_3+\alpha}(\pm\xi_1\xi_2\xi_3)
	x_{\gamma_1+\gamma_2+\gamma_3+2\alpha}(\pm\xi_1\xi_2\xi_3)
	\end{align*}
	we obtain that
	\begin{align*}
	g_1&=x_{[\gamma]+\alpha}\left(T_{[\gamma]\to[\gamma]+\alpha}(a)\right)
	\\
	&\times\,x_{\gamma_1+\gamma_2+\alpha}(\pm\xi_1\xi_2)
	x_{\gamma_1+\gamma_3+\alpha}(\pm\xi_1\xi_3)x_{\gamma_2+\gamma_3+\alpha}(\pm\xi_2\xi_3)\in H.
	\end{align*}
	Further from the equality
	\begin{align*}	
	[x_{-\alpha}(s),g_1]&=[x_{-\alpha}(s),x_{[\gamma]+\alpha}\left(T_{[\gamma]\to[\gamma]+\alpha}(a)\right)]
	\\
	&=x_{[\gamma]}(sa)x_{\gamma_1+\gamma_2+\alpha}(\pm s\xi_1\xi_2)x_{\gamma_1+\gamma_3+\alpha}(\pm s\xi_1\xi_3)x_{\gamma_2+\gamma_3+\alpha}(\pm s\xi_2\xi_3)\\
	&\times\,
	x_{\gamma_1+\gamma_2+\gamma_3+2\alpha}(\pm s\xi_1\xi_2\xi_3)x_{\gamma_1+\gamma_2+\gamma_3+\alpha}(\pm s^2\xi_1\xi_2\xi_3)
	\end{align*}
	we obtain that
	$$
	g_2=x_{[\gamma]}(sa)x_{\gamma_1+\gamma_2+\alpha}(\pm s\xi_1\xi_2)x_{\gamma_1+\gamma_3+\alpha}(\pm s\xi_1\xi_3)x_{\gamma_2+\gamma_3+\alpha}(\pm s\xi_2\xi_3)\in H.
	$$
	Similarly, from the equality
	\begin{align*}	
	[x_{\alpha}(t),g_2]&=[x_{\alpha}(t),x_{[\gamma]}(sa))]
		\\
		&=x_{[\gamma]+\alpha}\left(tsT_{[\gamma]\to[\gamma]+\alpha}(a)\right)x_{\gamma_1+\gamma_2+\alpha}(\pm ts^2\xi_1\xi_2) x_{\gamma_1+\gamma_3+\alpha}(\pm ts^2\xi_1\xi_3)
		\\
	&\times\,x_{\gamma_2+\gamma_3+\alpha}(\pm ts^2\xi_2\xi_3)
	x_{\gamma_1+\gamma_2+\gamma_3+\alpha}(\pm ts^3\xi_1\xi_2\xi_3)
	\\
	&\times\,x_{\gamma_1+\gamma_2+\gamma_3+2\alpha}(\pm t^2s^3\xi_1\xi_2\xi_3)
	\end{align*}
	we obtain that
	\begin{align*}
	x_{[\gamma]+\alpha}&\left(tsT_{[\gamma]\to[\gamma]+\alpha}(a)\right)x_{\gamma_1+\gamma_2+\alpha}(\pm ts^2\xi_1\xi_2)
	\\
	&\times\,x_{\gamma_1+\gamma_3+\alpha}(\pm ts^2\xi_1\xi_3)x_{\gamma_2+\gamma_3+\alpha}(\pm ts^2\xi_2\xi_3)\in H.
	\end{align*}
	
	The statement in question does not depend on the choice of signs in the Chevalley basis; hence we may assume for simplicity that $[e_{\alpha},e_{\gamma_i}]=e_{\gamma_i+\alpha}$. Then we have
	$$
	\big(s\xi_1,s\xi_2,s\xi_3,\pm s^2\xi_1\xi_2,\pm s^2\xi_1\xi_3,\pm s^2\xi_2\xi_3\big)\in A.
	$$
	Applying Lemma \ref{EliminateSquareTerm} we obtain that $(\xi_1,\xi_2,\xi_3,0,0,0)\in A$, which means exactly that $\zeta T_{[\gamma]\to[\gamma]+\alpha}(a)\in \ellev(H)_{[\gamma]+\alpha}$ for any $\zeta\in R$.
\end{proof}

For any prelevel $\sigma$ we set
$$
E(\sigma)=\<x_{[\gamma]}(a)\colon a\in\sigma_{[\gamma]} \>\le G(\Phi,R).
$$

Let us make the following observation.

\begin{lem}
\label{UniquenessOfALevel} Assume that for an almost level $\sigma$ we have 
	$$E(\sigma)\le H\le S(\sigma).$$
	 Then we have $\ellev(H)=\sigma$.
	\end{lem}
\begin{proof}
	It is clear that $\sigma\le\ellev(H)$. Let us prove the opposite inclusion. Let $\gamma\in\Phi\sm\Delta$. Let $a\in\ellev(H)_{[\gamma]}$. Take a root $\alpha\in\Delta$ such that $([\gamma],\alpha)=-1$. Since $H\le S(\sigma)$, we obtain that $\leftact{x_{\gamma}(a)}{e_{\alpha}}\in L_{\max}(\sigma)$. 
	In addition, we have
	$$
	\Big(\leftact{x_{[\gamma]}(a)}{e_{\alpha}}\Big)^{[\gamma]+\alpha}=T_{[\gamma]\to [\gamma]+\alpha}(a).
	$$
	Hence we have $T_{[\gamma]\to [\gamma]+\alpha}(a)\in \sigma_{[\gamma]+\alpha}$; and since any almost level is a prelevel, it follows that $a\in\sigma_{[\gamma]}$.
\end{proof} 

\subsection{Invariant level}

For every overgroup
$$
E(\Delta,R)\le H\le G(\Phi,R)
$$
 consider the set of all tandems $(g,l)$ such that $g\in H$. We denote by $L(H)$ the $R$-span of the corresponding elements $l\in L(\Phi,R)$. Let us make the following observation.

\begin{lem}
 \label{LofH} The submodule $L(H)\le L(\Phi,R)$ is a Lie subalgebra.
\end{lem}
\begin{proof}
	Let $(g_1,l_1)$ and $(g_2,l_2)$ be such tandems that $g_1$,$g_2\in H$. We must check that $[l_1,l_2]\in L(H)$. Note that the first component of the tandem $(\leftact{g_1}{g_2},\leftact{g_1}{l_2})$ is also contained in the subgroup $H$. Therefore, we obtain that $\leftact{g_1}{l_2}\in L(H)$.
	
	Applying Lemma \ref{tandemactpre} we obtain that
	$$
	\leftact{g_1}{l_2}=l_2+[l_1,l_2]-\zeta l_1
	$$
	for some $\zeta\in R$. Since $l_1$,$l_2\in L(H)$, we have $[l_1,l_2]\in L(H)$.
\end{proof}

Considering the tandems
 $(x_\alpha(\xi), \xi e_\alpha)$ for all the roots $\alpha\in\Delta$, we obtain that $L(H)$ is an $R$-overalgebra of the Lie algebra $L'(\Delta,R)$. We define the {\it invariant level} of an overgroup $H$ as follows:
$$
\invlev(H)=\lev(L(H)).
$$

Let us prove the following properties of the invariant level.

\begin{lem}
\label{EasyInclusion} The inclusion $\ellev(H)\le\invlev(H)$ holds true.
\end{lem}
\begin{proof}
	Let $\gamma\in\Phi\sm\Delta$. Let us prove that
	$$
	\ellev(H)_{[\gamma]}\le\invlev(H)_{[\gamma]}.
	$$
	Let $a\in\ellev(H)_{[\gamma]}$. Take a root $\alpha\in\Delta$ such that $([\gamma],\alpha)=-1$. Considering the tandem $(\leftact{x_{[\gamma]}(a)}{x_{\alpha}(1)}, \leftact{x_{[\gamma]}(a)}{e_{\alpha}})$, we obtain that $\leftact{x_{[\gamma]}(a)}{e_{\alpha}}\in L(H)$. In addition, we have 
	$$
	\left(\leftact{x_{[\gamma]}(a)}{e_{\alpha}}\right)^{[\gamma]+\alpha}=T_{[\gamma]\to [\gamma]+\alpha}(a).
	$$
	Hence we have
	$$
	T_{[\gamma]\to [\gamma]+\alpha}(a)\in \invlev(H)_{[\gamma]+\alpha};
	$$
	and since any almost level is a prelevel, it follows that $a\in\invlev(H)_{[\gamma]}$.
\end{proof}

\begin{lem}
\label{InclusionToS} The inclusion $H\le S(\invlev(H))$ holds true.
\end{lem}
\begin{proof}
	It is clear that $H\le\Stab_{G(\Phi,R)}(L(H))$. It remains to apply Corollary \ref{LargestStabilaser}.
\end{proof}

\section{What is going on?}
\label{SecWTF}

In this section we sum up all that was said previously, state the main result of the present paper, and describe a plan of its proof.

\smallskip

There are three conditions that occur in the statement of the main theorem. Two of them had already been formulated. Let us recall them.

Firstly, let us recall that if for subsystems $\Delta\le\Phi$ and $\Delta'\le\Phi'$ we have an embedding $\Phi'\hookrightarrow \Phi$ as a subsystem such that $\Delta\cap\Phi'=\Delta'$ and any root from $\Delta\sm\Delta'$ is orthogonal to any root from $\Phi'$, then we say that  ``the problem for the subsystem $\Delta\le\Phi$ includes the problem for the subsystem $\Delta'\le\Phi'$''.

\smallskip

The first condition concerns the subsystem and we always assume that it holds true.
\begin{equation}
	\begin{split}
		&\text{The problem for the subsystem } \Delta\le\Phi \text{ does not include}\\ &\text{ problems for subsystems of types } 
		\emp\le A_1 \text{ and } A_1\le A_2.
	\end{split}\tag{$*$}
\end{equation}

\smallskip

The second condition concerns the ground ring.
\begin{equation}
	\begin{split}
		\text{The ring } R \text{ does not have a residue field } \F_2.
	\end{split}\tag{$**$}
\end{equation}

\smallskip

Now we formulate the condition $(*{*}*)$, which is stronger that the condition $(*)$. For subsystems that satisfy $(*{*}*)$ we do not have to assume the condition $(**)$. 


\begin{equation}
	\begin{split}
		&\text{There are no blocks of three roots.} \text{ In addition, for any root } 
		\\ 
		&\qquad\gamma\in\Phi\sm\Delta \text{ there exist roots } \alpha_1,\alpha_2\in\Delta,
		\text{ such that }\\
		& \qquad\qquad(\gamma,\alpha_1)=(\gamma,\alpha_2)=-1 \text{ and } (\alpha_1,\alpha_2)=1. 
	\end{split}\tag{$*{*}*$}
\end{equation}

\smallskip

Now we want to describe overgroups $E(\Delta,R)\le H\le G(\Phi,R)$, where $\Delta\le\Phi$ is a subsystem that satisfies the condition $(*)$, and $R$ is a commutative ring. In order to do so we put in correspondence to any such an overgroup its elementary level $\ellev(H)$ and its invariant level $\invlev(H)$. Essential part of the present paper is the proof of the following proposition.

\begin{prop}
\label{EqualityOfLevels} Assume that either the condition $(**),$ or the condition $(*{*}*)$ holds true. Then for any overgroup $E(\Delta,R)\le H\le G(\Phi,R)$ we have the equality $\ellev(H)=\invlev(H)$.
\end{prop}

When the proof is done, this object will be called just the {\it level} of the overgroup~$H$, and denoted by $\lev(H)$.

\smallskip

Before Proposition \ref{EqualityOfLevels} is proved, a-priori we know the following.
\begin{itemize}
	\item Elementary level is a prelevel (Proposition \ref{EllevIsPrelev}).
	
	\item Invariant level is an almost level (by definition).
	
	\item The inclusions $E(\ellev(H))\le H\le S(\invlev(H))$ hold true (the first one by definition; the second one by Lemma \ref{InclusionToS}).
	
	\item If for an almost level $\sigma$ we have $E(\sigma)\le H\le S(\sigma)$, then we have $\ellev(H)=\sigma$ (Lemma \ref{UniquenessOfALevel}).
\end{itemize}

\smallskip

Let us introduce the following definition. 

\begin{defn}
 An almost level $\sigma$ is called a {\it level}, if the following inclusion holds true
$$
E(\sigma)\le S(\sigma).
$$
\end{defn} 
\begin{rem} ~In \S\ref{A2D4}, where we consider the problem for the subsystem \break $A_2\le D_4$ as a special case of the main result of the present paper, we give an example of an almost level that is not a level. 
\end{rem}

A-posteriori, after Proposition \ref{EqualityOfLevels} is proved, we deduce from it and from what is said above that under the conditions of Proposition \ref{EqualityOfLevels}
\begin{itemize}
	\item the level of an overgroup is a level.
\end{itemize}

\medskip

All what is said above allows us to obtain the main result of the present paper.

\begin{thm}
\label{sandwich} Let $\Phi$ be a root system of type $ADE$. Let $\Delta\le\Phi$ be its subsystem that satisfies the condition $(*)$. Let $R$ be a commutative ring. Assume that either the condition~$(**),$ or the condition~${(*{*}*)}$ hold true. Then for any overgroup $E(\Delta,R)\le H\le G(\Phi,R)$ there exists a unique level $\sigma$ such that
$$
E(\sigma)\le H\le S(\sigma).
$$
\end{thm}

Therefore, it remains to prove Proposition \ref{EqualityOfLevels}. Now we fix an overgroup $E(\Delta,R)\le H\le G(\Phi,R)$.

 The main part of the proof consists of two big parts that are the proofs of the following propositions.

\begin{prop}
\label{2A1prop}
 Assume that the condition $(**)$ holds true. Let $\gamma\in\Phi\sm\Delta,$ $\alpha_1,\alpha_2\!\in\!\Delta$ be such roots that $\alpha_1\perp\alpha_2$ and ${([\gamma],\alpha_1)\!=\!([\gamma],\alpha_2)\!=\!-1}$. Let $(g,l)$ be a tandem such that $g\in H$. Then 
 $l^{-\alpha_1}l^{[\gamma]}\in \ellev(H)_{[\gamma]}$.
\end{prop} 

\begin{prop}
\label{A2prop} If there is a block of three roots, then we assume that the con\-dition $(**)$ holds true. Let $\gamma\in\Phi\sm\Delta,$ $\alpha_1,$$\alpha_2\in\Delta$ be such roots that $(\alpha_1,\alpha_2)=1$ and $([\gamma],\alpha_1)=([\gamma],\alpha_2)=-1$. Let $(g,l)$ be a tandem such that $g\in H$. Then 
 $l^{-\alpha_1}l^{[\gamma]}\in \ellev(H)_{[\gamma]}$.
\end{prop}

From these two propositions, applying item 2 of Lemma \ref{Roots}, we obtain the following proposition.

\begin{prop}
\label{A1prop} Let $\gamma\in\Phi\sm\Delta,$ $\alpha\in\Delta$ be such roots that $([\gamma],\alpha)=-1$. Assume that one of the following conditions hold true:
\begin{enumerate}
	\item the condition $(**);$
	
	\item there are no blocks of three roots, and there exists a root $\alpha_2\in\Delta$ such that $([\gamma],\alpha_2)=-1,$ and $(\alpha,\alpha_2)=1$.
\end{enumerate}	
	 Let $(g,l)$ be a tandem such that $g\in H$. Then 
	 $l^{-\alpha}l^{[\gamma]}\in \ellev(H)_{[\gamma]}$.
	\end{prop}

Now we prove that Proposition \ref{A1prop} implies Proposition \ref{EqualityOfLevels}.

\begin{proof}
It is enough to show that for any root $\gamma\in\Phi\sm\Delta$ and any tandem $(g,l)$ such that $g\in H$ we have $l^{[\gamma]}\in \ellev(H)_{[\gamma]}$. 

Take such a root $\gamma$ and such a tandem $(g,l)$. Take a root $\alpha\in\Delta$ such that $([\gamma],\alpha)=-1$. If we are not assuming the condition $(**)$, then using the condition $(*{*}*)$, we take a root $\alpha$ in such a way that there exists a root $\alpha_2\in\Delta$, such that $([\gamma],\alpha_2)=-1$ and $(\alpha,\alpha_2)=1$.

We build the tandem
$$
(g_1,l_1)=\big(\leftact{g}{x_{\alpha}(1)},\leftact{g}{(e_{\alpha})}\big).
$$

Using Lemma \ref{tandemact} we obtain the following equalities:
\begin{align*}
l_1^{[\gamma]+\alpha}&=-T_{[\gamma]\to [\gamma]+\alpha}(l^{[\gamma]})-l^{-\alpha}l^{[\gamma]+\alpha},
&&l_1^{[\gamma]}=-l^{-\alpha}l^{[\gamma]},
\\
l_1^{-\alpha}&=-(l^{-\alpha})^2,
&&\  l_1^d=-l^{-\alpha}m,
\end{align*}
where $m=l^{\alpha}[e_{\alpha},e_{-\alpha}]+l^d$.

Further we build the tandem
$$
(g_2,l_2)=\big(\leftact{g_1}{x_{-\alpha}(1)},\leftact{g_1}{e_{-\alpha}}\big).
$$
Using Lemma \ref{tandemact} we obtain the following equalities:
$$
l_2^{[\gamma]}=-T_{[\gamma+\alpha]\to \gamma}(l_1^{[\gamma]+\alpha})+l_1^{\alpha}l^{-\alpha}l^{[\gamma]},\quad
l_2^{-\alpha}=1+l^{-\alpha}\xi+l_1^{\alpha}(l^{-\alpha})^2,\\
$$
where $\xi=[m,e_{-\alpha}]^{-\alpha}$.

Applying Proposition \ref{A1prop} to the tandem $(g_2,l_2)$ and roots $\gamma$ and $\alpha$, considering that
$$
T_{[\gamma+\alpha]\to \gamma}(l_1^{[\gamma]+\alpha})=-l^{[\gamma]}-l^{-\alpha}T_{[\gamma]+\alpha]\to[\gamma]}(l^{[\gamma]+\alpha}),
$$  
we obtain
\begin{align*}
l^{[\gamma]}&+l^{-\alpha}T_{[\gamma]+\alpha]\to[\gamma]}(l^{[\gamma]+\alpha})+l_1^{\alpha}l^{-\alpha}l^{[\gamma]}+l^{-\alpha} \xi l^{[\gamma]}+(l^{-\alpha})^2\xi T_{[\gamma]+\alpha]\to[\gamma]}(l^{[\gamma]+\alpha})
\\
&+l_1^{\alpha}(l^{-\alpha})^2\xi l^{[\gamma]}-l_1^{\alpha}(l^{-\alpha})^2T_{[\gamma+\alpha]\to \gamma}(l_1^{[\gamma]+\alpha})+(l_{1}^{\alpha})^2(l^{-\alpha})^3l^{[\gamma]}\in\ellev(H)_{[\gamma]}.
\end{align*}

It follows from Proposition \ref{A1prop} that 
$$
l^{-\alpha}l^{[\gamma]}\in \ellev(H)_{[\gamma]},
$$
 and that 
$$
l_{1}^{\alpha}l_{1}^{[\gamma]+\alpha}\in \ellev(H)_{[\gamma]+\alpha}.
$$
 Here we can apply Proposition~\ref{A1prop} to the roots $\gamma+\alpha$ and $-\alpha$ because of the symmetry. Hence we can remove the summands that are multiples of $l^{-\alpha}l^{[\gamma]}$ and $l_{1}^{\alpha}T_{[\gamma]+\alpha\to[\gamma]}(l_{1}^{[\gamma]+\alpha})$. Therefore, we obtain
$$
l^{[\gamma]}+l^{-\alpha}T_{[\gamma]+\alpha]\to[\gamma]}(l^{[\gamma]+\alpha})+(l^{-\alpha})^2\xi T_{[\gamma]+\alpha]\to[\gamma]}(l^{[\gamma]+\alpha})\in\ellev(H)_{[\gamma]}.\eqno{(\#)}
$$
 
Further it follows from Proposition \ref{A1prop} that $l^{\alpha}l^{[\gamma]+\alpha}\in\ellev(H)_{[\gamma]+\alpha}$. Hence we can multiply the sum in the statement $(\#)$ by $l^{\alpha}$, and then remove the summands that are multiples of $l^{\alpha}T_{[\gamma]+\alpha]\to[\gamma]}(l^{[\gamma]+\alpha})$. Therefore,
 $$l^{\alpha}l^{[\gamma]}\in \ellev(H)_{[\gamma]}.$$

We can repeat all the previous reasoning replacing $\gamma$ by $\gamma+\alpha$ and $\alpha$ by $-\alpha$, and obtain that
$$
l^{-\alpha}l^{[\gamma]+\alpha}\in\ellev(H)_{[\gamma]+\alpha}.
$$
 Hence we can remove from the statement $(\#)$ all the summands that are multiples of $l^{-\alpha}T_{[\gamma]+\alpha]\to[\gamma]}(l^{[\gamma]+\alpha})$. Thus we obtain
\begin{equation*}
l^{[\gamma]}\in \ellev(H)_{[\gamma]}.\qedhere
\end{equation*}
\end{proof}
Therefore, it remains to prove Propositions \ref{2A1prop} and \ref{A2prop}.

\section{A $2A_1$-proof}
\label{Sec2A1}

In this section we prove Proposition \ref{2A1prop}. The main idea is based on the remark after the proof of the main lemma in \cite{VavGav}, which tells that in order to get into a parabolic subgroup one can use an element that can be written as $x_\alpha(\xi)x_\beta(\zeta)$, where $\angle(\alpha,\beta)=\pi/2$. That way is called a $2A_1$-proof.

\subsection{Parabolic subgroups}

Let $\alpha_1$,$\alpha_2\in \Phi$ and $\alpha_1\perp\alpha_2$. We introduce a notation for the following linear functional on the span of $\Phi$:  
$$
\vpi_{\alpha_1,\alpha_2}(\gamma)=(\alpha_1+\alpha_2,\gamma).
$$
Since $\Phi$ is simply laced, it follows that for any $\gamma\in \Phi$ we have $\vpi_{\alpha_1,\alpha_2}(\gamma)\in\{-2,-1,0,1,2\}$. Note that the values $\pm 1$ may fail to occur. The set of all roots such that the value of $\vpi_{\alpha_1,\alpha_2}$ on them is nonnegative is a parabolic set. We denote by $P_{\alpha_1,\alpha_2}$ the corresponding parabolic subgroup, and we denote by $U_{\alpha_1,\alpha_2}$ its unipotent radical. 

Further we introduce the following notation:
\begin{align*}
\Sigma_{\alpha_1,\alpha_2}&=\{\gamma\in\Phi\colon\vpi_{\alpha_1,\alpha_2}(\gamma)=2\},
\\
U'_{\alpha_1,\alpha_2}&=\<x_{\gamma}(\xi)\colon\gamma\in\Sigma_{\alpha_1,\alpha_2},\, \xi\in R\>\le U_{\alpha_1,\alpha_2}.
\end{align*}
In other words, the group $U'_{\alpha_1,\alpha_2}$ coincides with the group $U_{\alpha_1,\alpha_2}$ if the functional $\vpi_{\alpha_1,\alpha_2}$ gives 3-grading on our root system (that is the values $1$ and $-1$ do not occur on roots). If the functional $\vpi_{\alpha_1,\alpha_2}$ gives a non-degenerate 5-grading, then the group $U'_{\alpha_1,\alpha_2}$ is a derived subgroup of the group $U_{\alpha_1,\alpha_2}$.

\subsection{Extraction of generators}
The following proposition allows to prove that something belongs to elementary level by finding certain elements in the group $H\cap U'_{\alpha_1,\alpha_2}$. 

\begin{prop}
\label{2A1Extraction} Let $\alpha_1,$ $\alpha_2\in \Delta$ and $\alpha_1\perp\alpha_2$. Let $\gamma_1,$$\ldots,$$\gamma_k\in\Sigma_{\alpha_1,\alpha_2}$ be pairwise non-equivalent roots. Let elements $a_i\in M_{[\gamma_i]}$ be such that
$$
g=\prod_{i=1}^{k} x_{[\gamma_i]}(a_i)\in H.
$$

Then we have $a_i\in\ellev(H)_{[\gamma_i]}$ for all $1\le i\le k$.
\end{prop}
\begin{proof}
	Assume that for some $g$ the converse is true. Among all such $g$ we choose the one with the minimal $k$. Then $a_i\notin\ellev(H)_{[\gamma_i]}$ for any $i$ because otherwise, since $U'_{\alpha_1,\alpha_2}$ is abelian, it follows that we can remove this factor, and make $k$ smaller. In particular, we have $\gamma_i\notin\Delta$.
	
	If $k=1$, then there is nothing to prove. Let $k\ge 2$. Since the roots $\gamma_i$ are pairwise non-equivalent, it follows that there exists a root $\beta\in\Delta$ such that $([\gamma_1],\beta)\ne ([\gamma_2],\beta)$.
	
	\smallskip
	
	\noindent{\bf Case 1:} $\vpi_{\alpha_1,\alpha_2}(\beta)=0$.
	
	One of the numbers $([\gamma_1],\beta)$ and $([\gamma_2],\beta)$ is nonzero; without loss of generality, we may assume that it is the first one. Further replacing, if necessary, $\beta$ by $-\beta$, we may assume that $([\gamma_1],\beta)=-1$. 
	
	Set
	$$
	g_1=[x_{\beta}(1),g]=\Big[x_{\beta}(1),\prod_{i=1}^k x_{[\gamma_i]}(a_i)\Big]= \prod_{i=1}^k [x_{\beta}(1), x_{[\gamma_i]}(a_i)].
	$$
	The last equality holds true because $x_{\beta}(1)\in P_{\alpha_1,\alpha_2}$, hence it normalises $U'_{\alpha_1,\alpha_2}$, and the group $U'_{\alpha_1,\alpha_2}$ is abelian. 
	 
	For any $i$ we have either $([\gamma_i],\beta)\ne -1$, then we obtain $[x_{\beta}(1), x_{[\gamma_i]}(a_i)]=e$; or $([\gamma_i],\beta)=-1$, then we obtain $|[\gamma_i]|=1$ (because otherwise we can add to the root $\beta$ two distinct roots from the block $[\gamma_i]$ and obtain a root with value of the functional $\vpi_{\alpha_1,\alpha_2}$ equal to 4), hence $[x_{\beta}(1), x_{[\gamma_i]}(a_i)]=x_{[\gamma_i]+\beta}(T_{[\gamma_i]\to[\gamma_i]+\beta}(a_i))$.
	
	Therefore, the first factor is equal to $x_{[\gamma_1]+\beta}(T_{[\gamma_1]\to[\gamma_1]+\beta}(a_1))$. Since we have $a_1\notin\ellev(H)_{[\gamma_1]}$ and $\ellev(H)$ is a prelevel, it follows that 
	$$
	T_{[\gamma_1]\to[\gamma_1]+\beta}(a_1)\notin\ellev(H)_{[\gamma_1]+\beta}.
	$$
	In addition, since the second factor is trivial, the number of non-trivial factors became less than $k$, which contradicts the assumption about minimality of $k$.
	
	\smallskip
	
	\noindent{\bf Case 2:} $\vpi_{\alpha_1,\alpha_2}(\beta)=\pm 2$.
	
	Replacing, if necessary, $\beta$ by $-\beta$, we may assume that $\vpi_{\alpha_1,\alpha_2}(\beta)=-2$. Since $\gamma_i\in\Sigma_{\alpha_1,\alpha_2}$ and $\gamma_i$ are distinct from the roots $\alpha_1$,$\alpha_2$, it follows that $([\gamma_i],\alpha_1)=([\gamma_i],\alpha_2)=1$. Hence, since $([\gamma_1],\beta)\ne ([\gamma_2],\beta)$, it follows that the root $\beta$ is distinct from the roots $-\alpha_1$ and $-\alpha_2$. Hence we have $(\beta,\alpha_1)=-1$ and $\beta+\alpha_1\in\Delta$. In addition, we have $([\gamma_1],\beta+\alpha_1)\ne([\gamma_2],\beta+\alpha_1)$ because $([\gamma_1],\beta)\ne([\gamma_2],\beta)$ and $([\gamma_1],\alpha_1)=([\gamma_2],\alpha_1)$. Therefore, replacing $\beta$ by $\beta+\alpha_1$ we reduce the problem to the first case.
	\smallskip
		
	\noindent{\bf Case 3:} $\vpi_{\alpha_1,\alpha_2}(\beta)=\pm 1$.
	
	One of the numbers $([\gamma_1],\beta)$ and $([\gamma_2],\beta)$ is nonzero; without loss of generality, we may assume that it is the first one. Further replacing, if necessary, $\beta$ by $-\beta$, we may assume that
	$
	([\gamma_1],\beta)=-1
	$.
	Then we have $\vpi_{\alpha_1,\alpha_2}(\beta)=-1$ because otherwise we would obtain $\vpi_{\alpha_1,\alpha_2}(\gamma_1+\beta)=3$. Then without loss of generality, we may assume that $(\beta,\alpha_1)=-1$ and $(\beta, \alpha_2)=0$.
	
	Set
	$$
	g_1=[x_{\beta}(1),g]=\Big[x_{\beta}(1),\prod_{i=1}^k x_{[\gamma_i]}(a_i)\Big]= \prod_{i=1}^k \Big[x_{\beta}(1), x_{[\gamma_i]}(a_i)\Big].
	$$
	The last equality holds true because all the factors in the right hand side belong to the subgroup $U_{\alpha_1,\alpha_2}$ and elements $x_{[\gamma_i]}(a_i)$ are central in this subgroup.
	
	For any $i$ we have either $([\gamma_i],\beta)\ne -1$, then we obtain $[x_{\beta}(1), x_{[\gamma_i]}(a_i)]=e$; or $([\gamma_i],\beta)=-1$, then we obtain $|[\gamma_i]|=1$ (because otherwise we can add to the root $\beta$ two distinct roots from the block $[\gamma_i]$  and obtain a root with value of the functional $\vpi_{\alpha_1,\alpha_2}$ equal to $3$), hence 
	$$
	[x_{\beta}(1), x_{[\gamma_i]}(a_i)]=x_{[\gamma_i]+\beta}(T_{[\gamma_i]\to[\gamma_i]+\beta}(a_i)).
	$$
	
	All the sums $\gamma_i+\beta$ that are roots belong to that set $\Sigma_{\alpha_1+\beta,\alpha_2}$. The first factor is equal to 
	$$
	x_{[\gamma_1]+\beta}(T_{[\gamma_1]\to[\gamma_1]+\beta}(a_1)).
	$$
	Since $a_1\notin\ellev(H)_{[\gamma_1]}$ and $\ellev(H)$ is a prelevel, it follows that 
	$$
	T_{[\gamma_1]\to[\gamma_1]+\beta}(a_1)\notin\ellev(H)_{[\gamma_1]+\beta}.
	$$
	In addition, since the second factor is trivial, the number of non-trivial factors became less than $k$, which contradicts the assumption about minimality of~$k$.
\end{proof}
 
 \subsection{Bitandems}
 
We need the following notion from the paper \cite{Gvoz2A1}. 
 
 \begin{defn}
 We call an element of the set $G(\Phi,R)\times L(\Phi,R)$ a {\it bitandem} if it can be written as
$$
\big(\leftact{h}{(x_{\alpha_1}(\xi)x_{\alpha_2}(\zeta))},\leftact{h}{(\xi e_{\alpha_1}+\zeta e_{\alpha_2})}\big),
$$
 where $\alpha_1$,$\alpha_2\in\Phi$, $\alpha_1\perp \alpha_2$, $\xi$,$\zeta\in R$ and $h\in G(\Phi,R)$.
\end{defn}
 	 
Similarly, We call a family $(g(t),l(t))\in G(\Phi,R)\times L(\Phi,R),$ $t\in R$ a {\it bitandem with parameter} if it can be written as
$$
 g(t)=\leftact{h}{((x_{\alpha_1}(t\xi))x_{\alpha_2}(t\zeta))},\quad
 l(t)=\leftact{h}{(t\xi e_{\alpha_1}+t\zeta e_{\alpha_2})}.
$$
 
In this case we use the notation $g=g(1)$, $l=l(1)$ for short.
 
 Further let $\alpha_1$,$\alpha_2\in \Phi$, $\alpha_1\perp \alpha_2$; and let $(g,l)$ be a tandem. Set
 $$
 (g_1(t),l_1(t))=\big(\leftact{g}{(x_{\alpha_1}(tl^{-\alpha_2})x_{\alpha_2}(- tl^{-\alpha_1}))},\leftact{g}{(tl^{-\alpha_2}e_{\alpha_1}- tl^{-\alpha_1}e_{\alpha_2})}\big).
 $$
 
A bitandem with parameter that can be obtain in such a way and the corresponding bitandem $(g_1,l_1)=(g_1(1),l_1(1))$ will be called {\it special} with respect to the pair $\alpha_1$, $\alpha_2$. 
 
 \medskip
 
The proofs of the following lemmas, which help perform calculations with bitandems, are written in the paper \cite{Gvoz2A1}.
 
 \begin{lem}
\label{bitandemact} Let $(g(t),l(t))$ be a bitandem with parameter, and let $v\in L(\Phi,R)$. Then there exists $w\in L(\Phi,R)$ such that for any $t\in R$ we have
$$
\leftact{g(t)}{v}=v+t[l,v]+t^2w.
$$
In addition, we have $2w=[l,[l,v]]$.
\end{lem}
 
 \begin{lem}
\label{bitandemsinparabolic} Let $(g_1(t),l_1(t))$ be a special with respect to the pair $\alpha_1, \alpha_2$ bitandem with parameter. Then $g_1(t)\in P_{\alpha_1,\alpha_2}$ for all $t\in R$.
\end{lem}
 
 From the proof of Lemma \ref{bitandemsinparabolic} one can obtain one more useful fact.
 
 \begin{lem}
\label{bitandemsinparabolicCor} Let $(g_1(t),l_1(t))$ be a special with respect to the pair $\alpha_1, \alpha_2$ bitandem with parameter that is constructed as above. Then the following equality holds true
 $$
 l_1=(l^{-\alpha_2}e_{\alpha_1}- l^{-\alpha_1}e_{\alpha_2})+[l,(l^{-\alpha_2}e_{\alpha_1}- l^{-\alpha_1}e_{\alpha_2})].
 $$
\end{lem}
 
 \subsection{The proof of Proposition \ref{2A1prop}}
 
 Let $\gamma\in\Phi\sm\Delta$, $\alpha_1$,$\alpha_2\in\Delta$ be such roots that $\alpha_1\perp\alpha_2$ and $([\gamma],\alpha_1)=([\gamma],\alpha_2)=-1$. Let $(g,l)$ be such tandem that $g\in H$. We must prove that $l^{-\alpha_1}l^{[\gamma]}\in \ellev(H)_{[\gamma]}$.
 
 We build the following bitandem with parameter that is special with respect to the pair $\alpha_1$,$\alpha_2$:
 $$
 (g_1(t),l_1(t))=\big(\leftact{g}{(x_{\alpha_1}(tl^{-\alpha_2})x_{\alpha_2}(- tl^{-\alpha_1}))},\leftact{g}{(tl^{-\alpha_2}e_{\alpha_1}- tl^{-\alpha_1}e_{\alpha_2})}\big).
 $$
 Further for any $t\in R$ we build the following tandem $$(g_{2,t},l_{2,t})=(\leftact{g_1(t)}{x_{\alpha_1}(1)},\leftact{g_1(t)}{e_{\alpha_1}}).$$
 
 Regardless of $t$, by Lemma \ref{bitandemsinparabolic} we have $g_1(t)\in P_{\alpha_1,\alpha_2}$; hence we obtain that $g_{2,t}\in H\cap U'_{\alpha_1,\alpha_2}$.
 
 Set
 $$
 g_{2,t}=\prod_{i=1}^{k} x_{[\gamma_i]}(a_{i,t}),\eqno{(1)}
 $$
 where $[\gamma_i]$ runs through all the blocks that are contained in $\Sigma_{\alpha_1,\alpha_2}$, and where $\gamma_1=\gamma+\alpha_1+\alpha_2$.
 
Let us prove the equality
 $$
 a_{1,t}=l_{2,t}^{[\gamma_1]}\eqno{(2)}
 $$
 Let us compute the coefficient $\left(\leftact{g_{2,t}}{e_{-\alpha_1}}\right)^{[\gamma]+\alpha_2}$ by two ways. On one hand, when we compute it directly using formula $(1)$, we obtain that $$\left(\leftact{g_{2,t}}{e_{-\alpha_1}}\right)^{[\gamma]+\alpha_2}=-T_{[\gamma_1]\to [\gamma]+\alpha_2}(a_{1,t}).$$ On the other hand, by Lemma \ref{tandemact} we have
 $$
 \leftact{g_{2,t}}{e_{-\alpha_1}}=e_{-\alpha_1}+[l_{2,t},e_{-\alpha_1}]+l_{2,t}^{\alpha_1}l_{2,t}.
 $$
 Since $g_1(t)\in P_{\alpha_1,\alpha_2}$, it follows that the element $l_{2,t}$ belongs to the sum of submodules $M_{[\gamma_i]}$. Hence $l_{2,t}^{[\gamma]+\alpha_2}=0$, which implies that
 $$
 \left(\leftact{g_{2,t}}{e_{-\alpha_1}}\right)^{[\gamma]+\alpha_2}=[l_{2,t},e_{-\alpha_1}]^{[\gamma]+\alpha_2}=-T_{[\gamma_1]\to [\gamma]+\alpha_2}(l_{2,t}^{[\gamma_1]}).
 $$
 Therefore, we proved that $a_{1,t}=l_{2,t}^{[\gamma_1]}$.
 
 Further by Lemma \ref{bitandemact} we have
 $$
 l_{2,t}^{[\gamma_1]}=(e_{\alpha_1}+t[l_1,e_{\alpha_1}]+t^2w)^{[\gamma_1]}=-tT_{[\gamma]+\alpha_2\to[\gamma_1]}(l_1^{[\gamma+\alpha_2]})+t^2 w^{[\gamma_1]}.
 $$
 Applying this equality, Equality~$(2)$, Proposition~\ref{2A1Extraction}, Lemma~\ref{EliminateSquareTerm}, and Proposition~\ref{EllevIsPrelev}, we obtain that $l_1^{[\gamma]+\alpha_2}\in\ellev(H)_{[\gamma]+\alpha_2}$.
 
 Further by Lemma \ref{bitandemsinparabolicCor} we have
 \begin{align*}
 l_1^{[\gamma]+\alpha_2}&=((l^{-\alpha_2}e_{\alpha_1}- l^{-\alpha_1}e_{\alpha_2})+[l,(l^{-\alpha_2}e_{\alpha_1}- l^{-\alpha_1}e_{\alpha_2})])^{[\gamma]+\alpha_2}
\\
&=l^{-\alpha_1}T_{[\gamma]\to [\gamma]+\alpha_2}(l^{[\gamma]}).
 \end{align*}
 Therefore, we have $l^{-\alpha_1}T_{[\gamma]\to [\gamma]+\alpha_2}(l^{[\gamma]})\in\ellev(H)_{[\gamma]+\alpha_2}$. Hence we obtain that $l^{-\alpha_1}l^{[\gamma]}\in\ellev(H)_{[\gamma]}$.
 
 \section{An $A_2$-proof}
 \label{SecA2}
 
In this section we prove Proposition \ref{A2prop}. The proof is similar to the one of Proposition \ref{2A1prop}; but now in order to get into a parabolic subgroup we use an element that can be written as $x_\alpha(\xi)x_\beta(\zeta)$, where $\angle(\alpha,\beta)=\pi/3$. Following \cite{VavGav}, we call this way an $A_2$-proof.
 
 \subsection{Parabolic subgroups}
 
 Let $\alpha_1,\alpha_2\in \Phi$ and $(\alpha_1,\alpha_2)=1$.  We introduce a notation for the following linear functional on the span of $\Phi$:  
 $$
 \vpi_{\alpha_1,\alpha_2}(\gamma)=(\alpha_1+\alpha_2,\gamma).
 $$
 Since $\Phi$ is simply laced, it follows that for any $\gamma\in \Phi$ we have $\vpi_{\alpha_1,\alpha_2}(\gamma)\in\{-3,-2,-1,0,1,2,3\}$. In addition, the values $3$ and $-3$ occur only on roots $\pm\alpha_1$ and $\pm\alpha_2$.  The set of all roots such that the value of $\vpi_{\alpha_1,\alpha_2}$ on them is nonnegative is a parabolic set. We denote by $P_{\alpha_1,\alpha_2}$ the corresponding parabolic subgroup, and we denote by $U_{\alpha_1,\alpha_2}$ its unipotent radical. 
 
 \subsection{Extraction of generators}
 
 The following proposition allows to prove that something belongs to elementary level by finding certain elements in the group $H\cap U_{\alpha_1,\alpha_2}$. 
 
 \begin{prop}
\label{A2Extraction}\! Let $\alpha_1,\alpha_2\!\in \!\Delta\!$ and $(\alpha_1,\alpha_2)\!=\!1$. Let $\gamma_1,\ldots,\gamma_k,$ $\delta_1$,$\ldots$,$\delta_n\!\in\!\Phi\!$ be pairwise non-equivalent roots such that ${\vpi_{\alpha_1,\alpha_2}(\gamma_i)\!=1\!}$ and $\vpi_{\alpha_1,\alpha_2}(\delta_i)=2$ for all $i$. Let elements $a_i\in M_{[\gamma_i]}$ and $b_i\in M_{[\delta_i]}$ be such that
\vspace{-4mm}
 	$$
 	g=\prod_{i=1}^{k} x_{[\gamma_i]}(a_i)\times\prod_{i=1}^{n} x_{[\delta_i]}(b_i)\in H.
 	$$
 	
 	Then we have $a_i\in\ellev(H)_{[\gamma_i]}$ for all $1\le i\le k$.
	\end{prop}
 
 \begin{proof} 	
 	We fix some complete order on the set of such blocks $[\gamma]$ that $\vpi_{\alpha_1,\alpha_2}(\gamma)>0$ in such a way that the blocks with value of the functional $\vpi_{\alpha_1,\alpha_2}$ equal to $1$ go first, after them go blocks with this value equal to $2$, and the blocks of the roots $\alpha_1$,$\alpha_2$ go last. Any element of the group $U_{\alpha_1,\alpha_2}$ can be written as a product of elementary generators in our order. We call this product a canonical form of an element.
 	
 	Let us begin the proof of the proposition. 	Assume that for some $g$ the converse is true. Among all such $g$ we choose the one with the minimal $k$. Then $a_i\notin\ellev(H)_{[\gamma_i]}$ for any $i$ because otherwise, as one can easily see, we could reduce $k$ (possibly increasing $n$). In particular, $\gamma_i\notin\Delta$.

	We will consider several cases. 
 	
 	\smallskip
 	
 \noindent	{\bf Case 1:} $k\ge 2$.
 	
 	Note that for any block $[\gamma_i]$ we have either $([\gamma_i],\alpha_1)=1$ and $([\gamma_i],\alpha_2)=0$, or vice versa. 
 	
 	\smallskip
 		
 \noindent	{\bf Subcase 1.1:} there are such indices $i\ne j$ that $([\gamma_i],\alpha_1)=1$ and $([\gamma_j],\alpha_2)=1$.
 	
 	\smallskip
 	
 	In this case we have $([\gamma_i],\alpha_2-\alpha_1)=-1$ and $([\gamma_j],\alpha_2-\alpha_1)=1$. Let $g_1$ be the element that we obtain if we take the canonical form of the commutator $[x_{\alpha_2-\alpha_1}(1),g]$ and remove the terms that correspond to roots $\alpha_1$ and $\alpha_2$. Cle\-arly, the element $g_1$ belongs to the subgroup $H$. In its canonical form the num\-ber of factors that correspond blocks with value of the functional $\vpi_{\alpha_1,\alpha_2}$ equal to $1$ is less than $k$; in addition, $x_{[\gamma_i]+\alpha_2-\alpha_1}(T_{[\gamma_i]\to [\gamma_i]+\alpha_2-\alpha_1}(a_i))$ is one of such factors. Since $a_i\notin \ellev(H)_{[\gamma_i]}$, it follows that $T_{[\gamma_i]\to [\gamma_i]+\alpha_2-\alpha_1}(a_i)\notin \ellev(H)_{[\gamma_i]+\alpha_2-\alpha_1}$. Therefore, we obtain a contradiction with the minimality of $k$.
 	
 	\smallskip
 		
 \noindent	{\bf Subcase 1.2:} there are no such indices as in Subcase 1.1. Without loss of generality, we may assume that $([\gamma_i],\alpha_1)=1$ for all $1\le i\le k$. 
 	
 	Since the roots $\gamma_i$ are pairwise non-equivalent, it follows that there exists a root $\beta\in\Delta$ such that $([\gamma_1],\beta)\ne ([\gamma_2],\beta)$. The assumptions we made imply that the root $\beta$ does not coincide with roots $\pm\alpha_1$, $\pm\alpha_2$, $\pm(\alpha_1-\alpha_2)$.
 	
 	\smallskip
 	
 \noindent	{\bf Subcase 1.2.1:} $\vpi_{\alpha_1,\alpha_2}(\beta)=0$.
 	
 	Replacing, if necessary, $\beta$ by $-\beta$ we may assume that one of the numbers $([\gamma_1],\beta)$ and $ ([\gamma_2],\beta)$ is equal to $-1$.
 	
 	Let $g_1$ be the element that we obtain if we take the canonical form of the commutator $[x_{\beta}(1),g]$ and remove the terms that correspond to roots $\alpha_1$ and $\alpha_2$. It is easy to see that replacing $g$ by $g_1$ we reduce~$k$.

	\smallskip
	
\noindent	{\bf Subcase 1.2.2:} $\vpi_{\alpha_1,\alpha_2}(\beta)=\pm 1$.
	
	Replacing $\beta$ by $-\beta$ we may assume that $\vpi_{\alpha_1,\alpha_2}(\beta)=1$.
	
	Further we may also assume that one of the numbers $([\gamma_1],\beta)$ and $ ([\gamma_2],\beta)$ is equal to $-1$. Indeed, we only have problem if one of these numbers is $0$ and the other is $1$; without loss of generality, we have $([\gamma_1],\beta)=0$ and $([\gamma_2],\beta)=1$. Note that in this case we have $(\beta,\alpha_1)=1$ because otherwise equalities $(\gamma_2,\beta)=1$, $(\gamma_2,\alpha_1-\alpha_2)=1$ and $(\beta,\alpha_1-\alpha_2)=-1$ would imply that $\gamma_2=\beta+(\alpha_1-\alpha_2)\in\Delta$. Therefore, we can achive what we want replacing $\beta$ by $\beta+(\alpha_2-\alpha_1)$.
	
	So we have $\vpi_{\alpha_1,\alpha_2}(\beta)\!=\!1$ and one of the numbers $([\gamma_1],\beta)$ and $([\gamma_2],\beta)$ is equal to~$-1$.
	
	Let $g_1$ be the element that we obtain if we take the canonical form of the commutator $[x_{-\beta}(1),[x_{\beta}(1),g]]$ and remove the terms that correspond to roots $\alpha_1$ and $\alpha_2$. It is easy to see that replacing $g$ by $g_1$ we reduce~$k$.
	
	\smallskip
	
\noindent	{\bf Subcase 1.2.3:} $\vpi_{\alpha_1,\alpha_2}(\beta)=\pm 2$.
	
	Replacing $\beta$ by $-\beta$ we may assume that $\vpi_{\alpha_1,\alpha_2}(\beta)=-2$.
	
	Further replacing $\beta$ by $\beta+\alpha_1$, we reduce the problem to Subcase 1.2.2.
	
	\smallskip
	
	\noindent{\bf Case 2:} $k=1$. Without loss of generality, we may assume that $([\gamma_1],\alpha_1)=1$.
	
	\smallskip
	
	\noindent{\bf Subcase 2.1:} $|[\gamma_1]|=1$.
	
	In this case we have:
	$$
	[x_{\alpha_2-\alpha_1}(1),g]=[x_{\alpha_2-\alpha_1}(1),x_{[\gamma_1]}(a_1)]=x_{[\gamma_1]+\alpha_2-\alpha_1}\left(T_{[\gamma_1]\to [\gamma_1]+\alpha_2-\alpha_1}(a_1)\right).
	$$
	Hence we obtain that $T_{[\gamma_1]\to [\gamma_1]+\alpha_2-\alpha_1}(a_1)\in\ellev(H)_{[\gamma_1]+\alpha_2-\alpha_1}$, and thus we have $a_1\in\ellev(H)_{[\gamma_1]}$.
	
	\smallskip
	
	\noindent{\bf Subcase 2.2:} $|[\gamma_1]|=2$.
	
	Let $[\gamma_1]=\{\gamma_1,\gamma'_1\}$ and $a_1=\xi e_{\gamma_1}+\xi' e_{\gamma'_1}$. Then 
	\begin{align*}
	[x_{\alpha_2-\alpha_1}(1),g]&=[x_{\alpha_2-\alpha_1}(1),x_{[\gamma_1]}(a_1)]
	\\
	&=x_{[\gamma_1]+\alpha_2-\alpha_1}\left(T_{[\gamma_1]\to [\gamma_1]+\alpha_2-\alpha_1}(a_1)\right)x_{\gamma_1+\gamma'_1+\alpha_2-\alpha_1}(\pm \xi\xi'),
	\end{align*}
	and also
	$$
	\big[x_{\alpha_1}(1),\big[x_{-\alpha_1}(1),[x_{\alpha_2-\alpha_1}(1),g]\big]\big]=x_{\gamma_1+\gamma'_1+\alpha_2-\alpha_1}(\pm \xi\xi').
	$$
	Hence we obtain that 
	$$
	T_{[\gamma_1]\to [\gamma_1]+\alpha_2-\alpha_1}(a_1)\in\ellev(H)_{[\gamma_1]+\alpha_2-\alpha_1},
	$$
	and $a_1\in\ellev(H)_{[\gamma_1]}$.

	\smallskip
	
	\noindent{\bf Subcase 2.3:} $|[\gamma_1]|=3$.
	
	Let us recall that if there is a block of three roots, then we assume the condition $(**)$.
	
	Let $[\gamma_1]=\{\gamma_1,\gamma'_1,\gamma''_1\}$ and $a_1=\xi e_{\gamma_1}+\xi' e_{\gamma'_1}+\xi'' e_{\gamma''_1}$. Then we have the equality 
	$
	\gamma_1+\gamma'_1+\gamma''_1+\alpha_2-\alpha_1=\alpha_1.
	$
It can be verified by computing the inner products of the left hand side with the roots $\alpha_1$ and $\alpha_2$.
	
	Further we set
	\begin{align*}
	A&=\big\{v\in R^6\colon\forall\,t\in R\, x_{\gamma_1}(tv_1)x_{\gamma'_1}(tv_2)x_{\gamma''_1}(tv_3)
	\\
	&\times x_{\gamma_1+\gamma'_1+\alpha_2-\alpha_1}(tv_4)x_{\gamma_1+\gamma''_1+\alpha_2-\alpha_1}(tv_5)
	 x_{\gamma'_1+\gamma''_1+\alpha_2-\alpha_1}(tv_6)\in H\big\}.
	\end{align*}
	The set $A$ is an additive subgroup in $R^6$ because the root elements in its definition commute modulo the root subgroup of the root $\alpha_1$, which is contained in the subgroup $H$. It is also clear that $A=RA$, i.e. it is an  $R$-submodule. 
	
	Take arbitrary elements $s$,$t\in R$.
	
	From the equality
	\begin{align*}
	&[x_{\alpha_2-\alpha_1}(s),g]=[x_{\alpha_2-\alpha_1}(s),x_{[\gamma_1]}(a_1)]
	\\
	&=x_{[\gamma_1]+\alpha_2-\alpha_1}(sT_{[\gamma_1]\to[\gamma_1]+\alpha_2-\alpha_1}(a_1))x_{\gamma_1+\gamma'_1+\alpha_2-\alpha_1}(\pm s\xi\xi')\\
	&\times x_{\gamma_1+\gamma''_1+\alpha_2-\alpha_1}(\pm s\xi\xi'')
	 x_{\gamma'_1+\gamma''_1+\alpha_2-\alpha_1}(\pm s\xi'\xi'')x_{\alpha_1}(\pm s\xi\xi'\xi'') x_{\alpha_2}(\pm s^2\xi\xi'\xi'')
	\end{align*}
	we obtain that
	\begin{align*}
	g_1&=x_{[\gamma_1]+\alpha_2-\alpha_1}(sT_{[\gamma_1]\to[\gamma_1]+\alpha_2-\alpha_1}(a_1))x_{\gamma_1+\gamma'_1+\alpha_2-\alpha_1}(\pm s\xi\xi')
	\\
	&\times  x_{\gamma_1+\gamma''_1+\alpha_2-\alpha_1}(\pm s\xi\xi'') x_{\gamma'_1+\gamma''_1+\alpha_2-\alpha_1}(\pm s\xi'\xi'')\in H.
	\end{align*}
	Further from the equality
	\begin{align*}
	[x_{\alpha_1-\alpha_2}(t),g_1]&=[x_{\alpha_1-\alpha_2}(t),x_{[\gamma_1]+\alpha_2-\alpha_1}(sT_{[\gamma_1]\to[\gamma_1]+\alpha_2-\alpha_1}(a_1))]
	\\
	&=x_{[\gamma_1]}(tsa_1)x_{\gamma_1+\gamma'_1+\alpha_2-\alpha_1}(\pm ts^2\xi\xi')x_{\gamma_1+\gamma''_1+\alpha_2-\alpha_1}(\pm ts^2\xi\xi'')
	\\
	&\times x_{\gamma'_1+\gamma''_1+\alpha_2-\alpha_1}(\pm ts^2\xi'\xi'')
	x_{\alpha_1}(\pm ts^3\xi\xi'\xi'') x_{\alpha_2}(\pm t^2s^3\xi\xi'\xi'')
	\end{align*}
	we obtain that
	\begin{align*}
	g_2&=x_{[\gamma_1]}(tsa_1)x_{\gamma_1+\gamma'_1+\alpha_2-\alpha_1}(\pm ts^2\xi\xi')
	\\
	&\times x_{\gamma_1+\gamma''_1+\alpha_2-\alpha_1}(\pm ts^2\xi\xi'')x_{\gamma'_1+\gamma''_1+\alpha_2-\alpha_1}(\pm ts^2\xi'\xi'')\in H.
	\end{align*}
	Therefore, for any $s\in R$ we obtained that
	$$
	(s\xi,s\xi',s\xi'',\pm s^2\xi\xi',\pm s^2\xi\xi'',\pm s^2\xi'\xi'')\in A.
	$$
	Applying Lemma \ref{EliminateSquareTerm} we obtain that $(\xi,\xi',\xi'',0,0,0)\in A$; in particular, we have $a_1\in \ellev(H)_{[\gamma_1]}$.
 \end{proof}

\subsection{$A_2$-tandems}

Let us give the following definition.

\begin{defn}
 We call an element of the set $G(\Phi,R)\times L(\Phi,R)$ an {\it $A_2$-tandem}, if it can be written as
$$
(\leftact{h}{(x_{\alpha_1}(\xi)x_{\alpha_2}(\zeta))},\leftact{h}{(\xi e_{\alpha_1}+\zeta e_{\alpha_2})}),
$$
 where $\alpha_1$,$\alpha_2\in\Phi$, $(\alpha_1,\alpha_2)=1$, $\xi$,$\zeta\in R$ and $h\in G(\Phi,R)$.
\end{defn}

The following property of $A_2$-tandems is required for our calculations.

  \begin{lem}
	\label{A2tandemact} Let $(g,l)$ be an $A_2$-tandem, and let $\beta\in\Phi$. Then
  	$$
  	\leftact{g}{e_\beta}=e_\beta+[l,e_\beta]-l^{-\beta} \cdot l.
  	$$
  \end{lem}
  \begin{proof}
  	It is enough to consider the universal situation, where $R=\Z[G][\xi,\zeta]$, $g=\leftact{g_{\gen}}{(x_{\alpha_1}(\xi)x_{\alpha_2}(\zeta))}$ and $l=\leftact{g_{\gen}}{(\xi e_{\alpha_1}+\zeta e_{\alpha_2})}$. Indeed, such an $A_2$-tandem can be specialised to any other. 
  	
  	The natural map $\Z[G][\xi,\zeta]\to \Z[G][\xi,\zeta,\zeta^{-1}]$ is injective. Hence it it enough to prove the equality in question for the ring $\Z[G][\xi,\zeta,\zeta^{-1}]$.
  	
  	It is easy to see that there exists such an element $h\in G(\Phi,Z[G][\xi,\zeta,\zeta^{-1}])$ that the equalities $x_{\alpha_1}(\xi)x_{\alpha_2}(\zeta)=\leftact{h}{x_{\alpha_1}(1)}$ and $\xi e_{\alpha_1}+\zeta e_{\alpha_2}=\leftact{h}{e_{\alpha_1}}$ hold true. Hence over the ring $\Z[G][\xi,\zeta,\zeta^{-1}]$ the pair $(g,l)$ is a tandem, and the equality in question holds true by Lemma~\ref{tandemact}.
  \end{proof}

	\begin{rem} In fact, the above argument shows something more general. Namely, if we consider the smallest closed subscheme (over $\Z$) of the scheme $G(\Phi,\,-\,)\times L(\Phi,\,-\,)$ that contains all tandems, then it will contain all $A_2$-tandems as well. This is not true for bitandems.
	\end{rem}
	
	 Further let $\alpha_1$,$\alpha_2\in \Phi$, $(\alpha_1,\alpha_2)=1$, and let $(g,l)$ be a tandem. Set
	$$
	(g_1,l_1)=\big(\leftact{g}{(x_{\alpha_1}(l^{-\alpha_2})x_{\alpha_2}(- l^{-\alpha_1}))},\leftact{g}{(l^{-\alpha_2}e_{\alpha_1}- l^{-\alpha_1}e_{\alpha_2})}\big).
	$$
	
	An $A_2$-tandem that can be obtained in such a way will be called {\it special} with respect to the pair $\alpha_1$, $\alpha_2$. By analogy with bitandems, we can prove the following lemma.
\begin{lem}
\label{A2tandemsinparabolic} Let $(g_1,l_1)$ be a special with respect to the pair $\alpha_1,$ $\alpha_2,$ $A_2$-tandem that is constructed as above. Then:
		\begin{enumerate}
			
			\item the following equality holds true
			$$
			l_1=(l^{-\alpha_2}e_{\alpha_1}- l^{-\alpha_1}e_{\alpha_2})+[l,(l^{-\alpha_2}e_{\alpha_1}- l^{-\alpha_1}e_{\alpha_2})];
			$$
			
			\item  $g_1\in P_{\alpha_1,\alpha_2}$.
		\end{enumerate}
\end{lem}
	\begin{proof}
		The first item can be checked by a direct calculation. Indeed, by Lemma~\ref{tandemact} we obtain
		\begin{align*}
		l_1&=l^{-\alpha_2}\cdot\leftact{g}{e_{\alpha_1}}-l^{-\alpha_1}\cdot\leftact{g}{e_{\alpha_2}}
		\\
		&=(l^{-\alpha_2}e_{\alpha_1}- l^{-\alpha_1}e_{\alpha_2})+[l,(l^{-\alpha_2}e_{\alpha_1}- l^{-\alpha_1}e_{\alpha_2})]+
		l^{-\alpha_2}l^{-\alpha_1}\cdot l-l^{-\alpha_1}l^{-\alpha_2}\cdot l
		\\
		&=(l^{-\alpha_2}e_{\alpha_1}- l^{-\alpha_1}e_{\alpha_2})+[l,(l^{-\alpha_2}e_{\alpha_1}- l^{-\alpha_1}e_{\alpha_2})].
		\end{align*}
		
		Now it is enough to check the second item in the universal situation, where $R=\Z[G][\xi]$, $g=\leftact{g_{\gen}}{x_{\alpha}(\xi)}$, $l=\leftact{g_{\gen}}{\xi e_{\alpha}}$.
		
		The natural map $\Z[G][\xi]\to\Z[G][\xi,(l^{-\alpha_2})^{-1}]$ is injective. Here we used that the ring $\Z[G]$ is a domain. Hence it is enough to prove the statement in question for the ring $\Z[G][\xi,(l^{-\alpha_2})^{-1}]$. 
		
As in the proof of Lemma~\ref{A2tandemact}, we obtain that over the ring $\Z[G][\xi,(l^{-\alpha_2})^{-1}]$ the pair $(g_1,l_1)$ is a tandem; hence we can apply Lemma~\ref{tandemsinparabolic}. In order to do so we must check that $l_1$ belong to the corresponding Lie-subalgebra, which follows from the first item.
	\end{proof}

\subsection{The proof of Proposition~\ref{A2prop}}

Let roots $\gamma\in\Phi\sm\Delta$, $\alpha_1$,$\alpha_2\in\Delta$ be such that $(\alpha_1,\alpha_2)=1$ and $([\gamma],\alpha_1)=([\gamma],\alpha_2)=-1$. Let $(g,l)$ be such a tandem that $g\in H$. We must prove that $l^{-\alpha_1}l^{[\gamma]}\in \ellev(H)_{[\gamma]}$.

We build the following $A_2$-tandem that is special with respect to the pair $\alpha_1$,$\alpha_2$: 
$$
(g_1,l_1)=(\leftact{g}{(x_{\alpha_1}(l^{-\alpha_2})x_{\alpha_2}(- l^{-\alpha_1}))},\leftact{g}{(l^{-\alpha_2}e_{\alpha_1}- l^{-\alpha_1}e_{\alpha_2})}).
$$
Then we build the following tandem $(g_2,l_2)=(\leftact{g_1}{x_{\alpha_1-\alpha_2}(1)},\leftact{g_1}{e_{\alpha_1-\alpha_2}})$.

By Lemma~\ref{A2tandemsinparabolic} we have $g_1\in P_{\alpha_1,\alpha_2}$; hence we obtain that $g_2\in H\cap P_{\alpha_1,\alpha_2}$.

Consider the following subsystems of the root system $\Phi$:
\begin{align*}
A_1&=\{\pm(\alpha_1-\alpha_2)\},
\\
A_2&=\{\pm\alpha_1,\pm\alpha_2,\pm(\alpha_1-\alpha_2)\},
\\
\Theta&=\{\beta\in\Phi\colon (\beta,\alpha_1)=(\beta,\alpha_2)=0\}.
\end{align*}
	
It is easy to see that  $\{\beta\in\Phi\colon \vpi_{\alpha_1,\alpha_2}(\beta)=0\}=\Theta\cup A_1$. This set is the symmetric part of the parabolic set of roots $\{\beta\in\Phi\colon \vpi_{\alpha_1,\alpha_2}(\beta)\ge 0\}$. Therefore, our elements can be written as $g_1=u_1h_1$ and $g_2=u_2h_2$, where $h_1$,$h_2\in T(\Phi,R)G(\Theta+A_1,R)$ and $u_1$,$u_2\in U_{\alpha_1,\alpha_2}$.

Note that the group scheme $G(\Theta+A_1,\,-\,)$ is a direct product of schemes $G(\Theta+A_1,\,-\,)=G(\Theta,\,-\,)\times G(A_1,\,-\,)$. Here we must check that the intersection of subschemes $G(\Theta,\,-\,)$ and $G(A_1,\,-\,)$ is trivial. In order to do so let us notice that the elements of subscheme $G(\Theta,\,-\,)$ commutes with the elements of the root subgroups $X_{\alpha_1}$ and $X_{\alpha_2}$; and it is easy to see that any element of the subscheme $G(A_1,\,-\,)$ that commutes with these subgroups is trivial.

Therefore, our elements can be written as $h_1=h'_1h''_1$, where $h'_1\in G(A_1,\,-\,)$ and $h''_1\in T(\Phi,R)G(\Theta,\,-\,)$. Hence we have $h_2=\leftact{h_1}{x_{\alpha_1-\alpha_2}(1)}=\leftact{h'_1}{x_{\alpha_1-\alpha_2}(\eps)}$ for some $\eps\in R^*$.

It is well known (see, for example, \cite{Taddei}) that the elementary subgroup $E(A_2,R)$ is normal in the group $G(A_2,R)$. Hence 
$$
h_2\in E(A_2,R)\le E(\Delta,R)\le H
$$
Thus we obtain that $u_2\in H$.

Let
\vspace{-3mm}
$$
u_2=\prod_{i=1}^{k} x_{[\gamma_i]}(a_i)\cdot\prod_{i=1}^{n} x_{[\delta_i]}(b_i)\cdot x_{\alpha_1}(\eta)x_{\alpha_2}(\theta),
$$
where $\gamma_1$,$\ldots$,$\gamma_k$,$\delta_1$,$\ldots$,$\delta_n\in\Phi$ are pairwise non-equivalent roots such that we have $\vpi_{\alpha_1,\alpha_2}(\gamma_i)=1$ and $\vpi_{\alpha_1,\alpha_2}(\delta_i)=2$ for all indices $i$. In addition we may assume that $[\gamma_1]=[\gamma]+\alpha_1$ and $[\gamma_2]=[\gamma]+\alpha_2$.

Then by Proposition~\ref{A2Extraction}, we have  $a_1\in \ellev(H)_{[\gamma_1]}$ and $a_2\in\ellev(H)_{[\gamma_2]}$.

The next step is to prove that $l_2^{[\gamma_1]}\in \ellev(H)_{[\gamma_1]}$.

It follows from the first item of Lemma~\ref{A2tandemsinparabolic} that the coefficients of the element $l_1$ in roots with negative value of the functional $\vpi_{\alpha_1,\alpha_2}$ vanish. Computing the element $l_2$ by Lemma~\ref{A2tandemact}, we obtain that such coefficient vanish for it as well.

Further we compute the element $(\leftact{g_2}{e_{-\alpha_1}})^{[\gamma]}$ by two ways.

 On one hand, computing it by Lemma~\ref{tandemact}, considering what is said above, we obtain that $(\leftact{g_2}{e_{-\alpha_1}})^{[\gamma]}=l_2^{[\gamma_1]}$. 

On the other hand, we have $h_2\in G(A_1,R)$; hence $\leftact{h_2}{e_{-\alpha_1}}=\xi e_{-\alpha_1}+\zeta e_{-\alpha_2}$ for some $\xi$,$\zeta\in R$. Therefore, 
\begin{align*}
(\leftact{g_2}{e_{-\alpha_1}})^{[\gamma]}&=\left(\leftact{u_2}{(\xi e_{-\alpha_1}+\zeta e_{-\alpha_2})}\right)^{[\gamma]}
\\
&=-\xi T_{[\gamma_1]\to[\gamma]}(a_1)-\zeta T_{[\gamma_2]\to[\gamma]}(a_2)\in \ellev(H)_{[\gamma]}.
\end{align*}

So we have $l_2^{[\gamma_1]}\in \ellev(H)_{[\gamma_1]}$. Similarly we have $l_2^{[\gamma_2]}\in \ellev(H)_{[\gamma_2]}$.

Further computing the element $l_2^{[\gamma_1]}$ by Lemma~\ref{A2tandemact} we obtain that
$$
l_2^{[\gamma_1]}=-T_{[\gamma_2]\to [\gamma_1]}(l_1^{[\gamma_2]})-l_1^{\alpha_2-\alpha_1} l_1^{[\gamma_1]}.
$$
Therefore, $T_{[\gamma_2]\to [\gamma_1]}(l_1^{[\gamma_2]})+l_1^{\alpha_2-\alpha_1} l_1^{[\gamma_1]}\in \ellev(H)_{[\gamma_1]}$. Hence 
 $$
 l_1^{[\gamma_2]}+l_1^{\alpha_2-\alpha_1} T_{[\gamma_1]\to[\gamma_2]}(l_1^{[\gamma_1]}) \in \ellev(H)_{[\gamma_2]}.\eqno(1)
 $$

Similarly considering the tandem $(\leftact{g_1}{x_{\alpha_2-\alpha_1}(1)},\leftact{g_1}{e_{\alpha_2-\alpha_1}})$ we obtain that

$$
T_{[\gamma_1]\to [\gamma_2]}(l_1^{[\gamma_1]})+l_1^{\alpha_1-\alpha_2} l_1^{[\gamma_2]}\in \ellev(H)_{[\gamma_2]}.\eqno(2)
$$

Computing the element $l_2^{[\gamma_2]}$ by Lemma~\ref{A2tandemact} we obtain that
$$
l_2^{[\gamma_2]}=-l_1^{\alpha_2-\alpha_1}l_1^{[\gamma_2]}.
$$

Hence we also have
$$
l_1^{\alpha_2-\alpha_1}l_1^{[\gamma_2]}\in\ellev(H)_{[\gamma_2]}.\eqno(3)
$$

Applying (1), (2) and (3) we obtain that
\begin{align*}
l_1^{[\gamma_2]}=\left(l_1^{[\gamma_2]}+l_1^{\alpha_2-\alpha_1} T_{[\gamma_1]\to[\gamma_2]}(l_1^{[\gamma_1]})\right)&-l_1^{\alpha_2-\alpha_1}\left(T_{[\gamma_1]\to [\gamma_2]}(l_1^{[\gamma_1]})+l_1^{\alpha_1-\alpha_2} l_1^{[\gamma_2]}\right)
\\
&+l_1^{\alpha_1-\alpha_2}\left(l_1^{\alpha_2-\alpha_1}l_1^{[\gamma_2]} \right)\in \ellev(H)_{[\gamma_2]}.
\end{align*}

Finally, computing the element $l_1^{[\gamma_2]}$ by the first item of Lemma~\ref{A2tandemsinparabolic} we obtain that $l_1^{[\gamma_2]}=l^{-\alpha_1}T_{[\gamma]\to [\gamma_2]}(l^{[\gamma]})$. Thus, we have $l^{-\alpha_1}T_{[\gamma]\to [\gamma_2]}(l^{[\gamma]})\in \ellev(H)_{[\gamma_2]}$. Hence $l^{-\alpha_1}l^{[\gamma]}\in\ellev(H)_{[\gamma]}$.

\section{Special cases and corollaries for multiply laced systems}
\label{SecSpecialCases}

In this section we consider the most interesting special cases of Theorem \ref{sandwich}. We also show that for some subsystems of multiply laced systems the description of overgroups can be obtained as a corollary of Theorem \ref{sandwich}.

\setcounter{subsection}{-1}

\subsection{Ideal levels}

A collection of ideals $\sigma=\{\sigma_\alpha\}_{\alpha\in\Phi}$ of the ring $R$ is called a {\it net of ideals}, if the following conditions hold true.
\begin{enumerate}
	\item If $\alpha$, $\beta$, $\alpha+\beta\in\Phi$, then $\sigma_\alpha\sigma_\beta\sub \sigma_{\alpha+\beta}$.
	
	\item If $\alpha\in\Delta$, then $\sigma_\alpha=R$.
\end{enumerate} 

Assume that the subsystem $\Delta\le\Phi$ is such that every block has one root. Then it is easy to see that almost levels are in one-to-one correspondence with nets of ideals.

If we identify an almost level with the corresponding net of ideals, then the algebra $L_{\max}(\sigma)$ defined in the present paper consides with the algebra $L(\sigma)$, defined in \cite{Gvoz2A1}:
$$
L_{\max}(\sigma)=L(\sigma)=D\oplus\bigoplus_{\alpha\in\Phi}\sigma_\alpha e_\alpha\le L(\Phi,R).
$$

Hence the group $S(\sigma)$ defined in the present paper consides with the group $S(\sigma)$ defined in \cite{Gvoz2A1}. Let us also note that in this case any almost level is a level, see \cite[Corollary 1]{Gvoz2A1}. 

Therefore, the first part of the main result in \cite{Gvoz2A1} is a special case of the main result of the present paper. 

\subsection{$\textbf{D}_{n-1}\le \textbf{D}_n$}

Consider the subsystem $D_{n-1}\le D_n$, where $n\ge 3$. It is easy to see that the condition $(*)$ holds true. In addition, if $n\ge 4$, then the condition $(*{*}*)$ hold true. All the blocks whose roots does not belong to the subsystem consists of two roots. In addition, between any two such blocks there is a path with every step being an addition of root from the subsystem. Hence in this case a prelevel can be uniquely recovered from one its component. 

We recall that the algebra $L(D_n,R)=\mathfrak{so}(2n,R)$ can be defined as the algebra of all such $(2n\times 2n)$ matrices~$l$ over the ring $R$ that $fl+l^tf=0$, where
$$
f=\begin{pmatrix}
	&       & 1\\
	&\reflectbox{$\ddots$} &  \\
	1 &       &
\end{pmatrix}
$$

We get used to numerate rows and columns of these matrices as follows: 
$$
1,\ldots,n,-n,\ldots,-1.
$$
 We denote by $e_{i,j}$ the corresponding matrix units.

In this case the set of matrices $\{e_{i,j}-e_{-j-i}\colon i\ne\pm j\}$ coincides with the set of elements $\{\pm e_{\alpha}\colon \alpha\in D_n\}$. In addition, if $i,j\ne \pm n$, then the corresponding root belongs to $D_{n-1}$.

We choose signs in such a way that matrices $e_{1,n}-e_{-n,-1}$ and $e_{1,-n}-e_{n,-1}$ were root elements and not the opposite to them. We denote by $\delta$ and $\delta'$ the corresponding roots. Note that these roots form a block. 

Elements $e_{\delta}$ and $e_{\delta'}$ form a basis of the module $M_{[\delta]}$. Using this basis we identify the module $M_{[\delta]}$ with $R^2$. Therefore, to every prelevel we put in correspondence a submodule in $R^2$; and the prelevel can be uniquely recovered from this submodule. We identify a prelevel with the corresponding submodule.

Further, if $n=3$, then we assume that the ring $R$ satisfies the condition $(**)$. Let us show that, firstly, any submodule in $R^2$ occur as a prelevel; and secondly, every prelevel is a level. In order to do so we present an overgroup $E(D_{n-1},R)\le H\le G(D_n,R)$ such that its level is given by a prescribed submodule $A\le R^2$. In addition, since the definitions of prelevel and level do not depend on the choice of a group in the given isogeny class, it follows that we can choose the most convenient one. That would be the special orthogonal group $G(D_n,R)=\mathrm{SO}(2n,R)$, which is defined as the group of $2n\times 2n$ matrices that preserve the following quadratic form
$$
x_1x_{-1}+\ldots+x_{n}x_{-n},
$$
and have trivial Dickson invariant. In this case we have
\begin{align*} 
&x_{\delta}(\xi)=e+\xi(e_{1,n}-e_{-n,-1}),\\
&x_{\delta'}(\zeta)=e+\zeta(e_{1,-n}-e_{n,-1}),\\
&x_{[\delta]}(\xi,\zeta)=x_{\delta}(\xi)x_{\delta'}(\zeta)=e+\xi(e_{1,n}-e_{-n,-1})+\zeta(e_{1,-n}-e_{n,-1})-\xi\zeta e_{1,-1},
\end{align*}
where $e$ is the identity matrix.

We fix a submodule $A\le R^2$, and set
$$
\mathfrak{I}=\{a\in M(2n,R)\colon \forall i,\; (a_{i,n},a_{i,-n})\in A\}.
$$
It is easy to see that $\mathfrak{I}$ is a left ideal of the matrix ring $M(2n,R)$. Hence the set
$$
H=\mathrm{SO}(2n,R)\cap (e+\mathfrak{I})
$$
is a subgroup of the group $\mathrm{SO}(2n,R)$. It follows from the definition that $x_{[\delta]}(\xi,\zeta)\in H$ if and only if $(\xi,\zeta)\in A$. Therefore, 
$$
\lev(H)=\ellev(H)=A.
$$

So we establish that levels correspond to submodules in $R^2$. Consider the overgroup $E(D_{n-1},R)\le G(B_{n-1},R)\le G(D_n,R)$. Here the first embedding is an embedding of a subsystem subgroup; and the second one is obtained by folding. The level of this overgroup is the diagonal in $R^2$. The same is true for the group $E(B_{n-1},R)$.

If an overgroup $E(D_{n-1},R)\le H\le G(D_n,R)$ is contained in the group $G(B_{n-1},R)$, then its level is contained in the diagonal; and such submodules are in one-to-one correspondence with ideals of the ring $R$. Therefore, overgroups $E(D_{n-1},R)$ in $G(B_{n-1},R)$ are described by ideals of the ring $R$. Similarly overgroups $E(B_{n-1},R)$ in $G(D_n,R)$ are described by ideals of the ring $R$.

All these results for orthogonal groups are special cases of the result of the paper by E.~Voronetsky \cite{VoronNormilised} on odd unitary groups. It is said explicitly in this paper about overgroups $\mathrm{EO}(2n-2,R)$ and $\mathrm{O}(2n-1,R)$, and about overgroups $\mathrm{EO}(2n-1,R)$ in $\mathrm{O}(2n,R)$. It is not mentioned in \cite{VoronNormilised} that the description of overgroups $\mathrm{EO}(2n-2,R)$ in $\mathrm{O}(2n,R)$ is a special case; however its author told me that it is true. 

\begin{rem}
 In a general case, when there are blocks of two or three roots, one can ask whether it is true that any submodule in $R^2$ (resp. in $R^3$) can be a component of a prelevel. In other words, is it true that we cannot walking on blocks and applying $T$-operators go in a circle and change sign at one of the components without changing sign at the other? This would mean that the components of a prelevel must be stable with respect to this operation. The answer is yes.
	
It can be shown that for any block of two roots the problem for the subsystem $\Delta\le\Phi$ includes the problem for a subsystem $\Delta'\le D_n$ such that the subsystem $\Delta'$ is contained in $D_{n-1}$, and this block is a block with respect to the subsystem $D_{n-1}\le D_n$. Then it follows from the reasoning above that the described above situation with signs is impossible. 

Similarly for any block of three roots the problem for the subsystem $\Delta\le\Phi$ includes the problem for a subsystem $A_2\le D_4$ and this block is a block with respect to this subsystem. Considering the overgroup $E(A_2,R)\le G(G_2,R)\le G(D_4,R)$, we obtain that the diagonal can be a component of a level. Hence here the described above situation with signs is impossible as well. 
\end{rem}

\subsection{$\textbf{4A}_1\le\textbf{E}_6$}

Consider the subsystem $4A_1\le E_6$. It is easy to see that the condition $(*)$ holds true. Blocks that have roots not in the subsystem form 7 orbits with respect to the action of the Weyl group $W(4A_1)$. One of these orbits contains blocks of one root, and 6 orbits contains blocks of two roots. 

Therefore, a prelevel is given by a collection of one ideal and six submodules in $R^2$. In order to be an almost level or a level such a collection must satisfy certain conditions that we will not try to write down explicitly. However, in order to describe overgroups of $E(4A_1,R)$ in $G(F_4,R)$ it is enough to consider only such levels that their components in blocks of two roots are contained in the diagonal. Almost levels with such property are in natural one-to-one correspondence with collections of ideals $\sigma=\{\sigma_\alpha\}_{\alpha\in F_4}$ of the ring $R$ that satisfy the following conditions:
\begin{enumerate}
	\item if $\alpha$,$\beta\in F_4$ and $\angle(\alpha,\beta)>\tfrac{\pi}{2}$, then $\sigma_\alpha\sigma_\beta\sub \sigma_{\alpha+\beta}$;
	
	\smallskip
	
	\item if $\alpha$,$\beta\in F_4$ are such orthogonal short roots that $\alpha+\beta\in F_4$, then
	$2\sigma_\alpha\sigma_\beta\sub \sigma_{\alpha+\beta}$;
	
	\smallskip
	
	\item if $\alpha\in 4A_1$, then $\sigma_\alpha=R$.
\end{enumerate} 

We identify almost levels with the corresponding collections of ideals. Assume, for simplicity, that $2\in R^*$. Then such collections are nets of ideals. It is easy to see that in this case all of them are levels. In addition, the intersection of the corresponding subgroup $S(\sigma)$ with the subgroup $G(F_4,R)$ coincides with the subgroup $S(\sigma)$ defined in the paper \cite{GvozInside}. In this paper it is shown how one can enlarge the subgroup $E(\sigma)$ so that the new subgroup $\hat{E}(\sigma)$ were normal in $S(\sigma)$.

For the description of overgroups of $D_4$ in $F_4$ we make the following remarks. Firstly, i this case all the almost levels are levels even if we do not require the invertability of two. Secondly, we have $\hat{E}(\sigma)=E(\sigma)$.

Note that as a corollary we obtain the description of overgroups of $E(F_4,R)$ in $E(E_6,R)$. It were obtained earlier in the paper \cite{LuzF4E6} by A.~Luzgarev without any assumptions on the ground ring. In our proof we also do not require the condition $(**)$ because the subsystem $D_4\le E_6$ satisfies the condition $(*{*}*)$.

\subsection{$\textbf{lA}_1\le\textbf{A}_{2l-1}$}

Consider the subsystem $lA_1\le A_{2l-1}$, where $l\ge 2$. It is easy to see that the condition $(*)$ holds true. All the blocks whose roots does not belong to the subsystem consists of two roots. The number of $W(lA_1)$-orbits of such blocks is equal to $\binom{l}{2}$. Hence a prelevel is defined by $\binom{l}{2}$ submodules in $R^2$. In order to be an almost level or a level such submodules must satisfy certain conditions that we will not try to write down explicitly. However, in order to describe overgroups of $E(lA_1,R)$ in $G(C_l,R)$ it is enough to consider only such levels that their components in blocks of two roots are contained in the diagonal. Almost levels with such property are in natural one-to-one correspondence with nets of ideals for the system $C_l$. Any such an almost level is a level, and one can apply the results of \cite{GvozInside} to the corresponding sandwich.

Note that we also obtained the description of overgroups of
$$
E(C_l,R)\mbox{ in } G(A_{2l-1},R).
$$
They are described by ideals of the ring $R$. For the simply connected groups this result was obtained earlier by N.~A.~Vavilov and V.~A.~Petrov in the paper~\cite{VP-Ep}. In addition, they require the condition $(**)$ only for the case $l=2$. 

Note that over a field the description of overgroups for the subsystem $lA_1\le A_{2l-1}$ in case of simply connected groups was obtained in the paper~\cite{KoibaevBlockdiag} by V.~A.~Koibaev.

\subsection{$\textbf{A}_2\le\textbf{D}_4$}
\label{A2D4}

Consider the subsystem $A_2\le D_4$. It is easy to see that the condition $(*)$ holds true.  All the blocks whose roots does not belong to the subsystem consists of three roots. There are two $W(A_2)$-orbits of such blocks. Hence a prelevel is defined by two submodules in $R^3$. In order to be an almost level or a level such submodules must satisfy certain conditions that we will not try to write down explicitly. However, in order to describe overgroups of $E(A_2,R)$ in $G(G_2,R)$ it is enough to consider only such levels that their components in blocks of three roots are contained in the diagonal. Almost levels with such property are in natural one-to-one correspondence with pairs of ideals $A$,$B\unlhd R$ that satisfy the conditions $2A^2\le B$ and $2B^2\le A$. In order to be a level such a pair must satisfy the following stronger condition: the ideal $A$ must contain the squares of all the elements of the ideal $B$ and vice versa. 

One can give an example of a pair of ideals that defines an almost level but not a level. For example assume that $2=0$ in the ring $R$ and take $A=R$ and $B=(0)$. However, if $2\in R^*$, then every almost level is a level and one can apply the results of the paper \cite{GvozInside} to the corresponding sandwich.

Note that we also obtained the description of overgroups of $E(G_2,R)$ in $G(D_4,R)$. They are described by pairs of submodules in $R^2$ that satisfy certain conditions.

\section{Cases that are not done}

Using the tables from the paper \cite{VSch} we state the problems for multiply laced systems such that the description of overgroups for them cannot be obtained as a corollary of the main result of the present paper. Firstly, there are three problems for the full rank subsystems. 

\begin{prob}
 Describe overgroups of $E(A_2+\tilde{A_2},R)$ in $G(F_4,R)$.
\end{prob}

\begin{prob}
 Describe overgroups of $E(A_3+\tilde{A_1},R)$ in $G(F_4,R)$.
\end{prob}

\begin{prob}
 Describe overgroups of $E(A_1+\tilde{A_1},R)$ in $G(G_2,R)$.
\end{prob}

Here tilda stands for the short root component. It is most likely that the first two problems can be solved by the same means that are used in the present paper. Namely one can use $A_2$ of long roots in order to get into the parabolic subgroup. The third problem can be more complicated. 
 
Secondly, the cases for subsystems $\Delta\le \Phi$, where $\Phi$ is a doubly laced systems and there is a short root in $\Phi\cap\Delta^\perp$ that can be added to the short root from $\Delta$ so that the sum was a root, remain unstudied. In such problems one have to assume that $2\in R^*$. There is one such problems for the exceptional groups.

\begin{prob}
  Describe overgroups of $E(B_3,R)$ in $G(F_4,R)$.
\end{prob}

There are also such subsystems in $B_n$.

\begin{prob}
  Describe overgroups of $E(B_{n-1},R)$ in $G(B_n,R)$, where $n\ge 3$. For which subsystem smaller than $B_{n-1}$ the description is possible\textup? 
\end{prob}

\end{document}